\documentclass[12pt]{article}
\usepackage{setspace}

\usepackage[]{graphicx,float,latexsym,times}
\usepackage{amsfonts,amstext,amsmath,amssymb,amsthm}
\usepackage{caption2}
\usepackage{multirow}
\usepackage{subfigure}
\usepackage{float}

\usepackage[utf8]{inputenc} 
\usepackage[T1]{fontenc}
\usepackage{url}              
\usepackage{cite}             
\usepackage{booktabs}         
\usepackage[hidelinks]{hyperref}
\usepackage{nicefrac}       
\usepackage{microtype}      
\usepackage{mathtools}
\usepackage{algorithmic}
\usepackage{textcomp}
\usepackage{xcolor}

\usepackage{mathrsfs}
\usepackage{multirow,comment}
\usepackage{bbm}
\usepackage{bm}
\usepackage{enumitem}
\usepackage{array}

\usepackage{ifthen}

\usepackage[normalem]{ulem}

\newtheorem{innercustomlem}{Lemma}
\newenvironment{customlem}[1]
  {\renewcommand\theinnercustomlem{#1}\innercustomlem}
  {\endinnercustomlem}

\long\def\symbolfootnote[#1]#2{\begingroup%
\def\thefootnote{\fnsymbol{footnote}}\footnote[#1]{#2}\endgroup}

\usepackage[margin=1in]{geometry}

\usepackage{titlesec} 

\titleformat{\section}{\large\bfseries}{\thesection.}{.5em}{}
\titlespacing*{\section}{0pt}{*3}{*2}
\titleformat{\subsection}{\normalfont\bfseries}{\thesubsection.}{.5em}{}
\titlespacing*{\subsection} {0pt}{*3}{*2}
\titleformat{\subsubsection}{\normalfont\bfseries}{\thesubsubsection.}{.5em}{}
\titlespacing*{\subsubsection} {0pt}{*3}{*2}

\theoremstyle{plain} 
\newtheorem{theorem}{Theorem}[section]

\newtheorem{lemma}{Lemma}[section]
\newtheorem{corollary}{Corollary}[section]

\theoremstyle{definition} 
\newtheorem{definition}{Definition}[section]
\newtheorem{remark}{Remark}[section]
\newtheorem{example}{Example}[section]

\numberwithin{equation}{section} 

\usepackage[authoryear]{natbib}

\newcommand{\yuchen}[1]{{\color{blue}{#1}}}

\newcommand{\liyan}[1]{{\color{violet}{#1}}}

\newcommand{\brc}[1]{\left( #1 \right)}
\newcommand{\sbrc}[1]{\left[ #1 \right]}
\newcommand{\cbrc}[1]{\left\{ #1 \right\}}

\newcommand{\CADD}{{\mathsf{CADD}}}
\newcommand{\WADD}{{\mathsf{WADD}}}

\providecommand{\cF}{\mathcal{F}}
\providecommand{\cX}{\mathcal{X}}
\providecommand{\cC}{\mathcal{C}}

\providecommand{\hatP}{\widehat{P}}
\providecommand{\scrP}{\mathscr{P}}
\providecommand{\scrX}{\mathscr{X}}
\providecommand{\scrY}{\mathscr{Y}}
\providecommand{\calP}{\mathcal{P}}
\providecommand{\calQ}{\mathcal{Q}}

\providecommand{\wass}{\mathsf{W}}

\providecommand{\ind}{\mathbb{I}}
\newcommand{\norm}[1]{\left\lVert #1 \right\rVert}
\DeclareMathOperator*{\esssup}{ess\,sup}
\DeclareMathOperator*{\argmax}{arg\,max}
\DeclareMathOperator*{\argmin}{arg\,min}
\DeclareMathOperator*{\erf}{erf}

\def\bE{{\mathbb{E}}}
\def\bN{{\mathbb{N}}}

\begin{document}

\title{\textbf{\Large Distributionally Robust Quickest Change Detection using \\ Wasserstein Uncertainty Sets}}

\date{}
\maketitle

\author{
\begin{center}
\vskip -1cm

\textbf{\large Liyan Xie$^{\mathrm{a}*}$, Yuchen Liang$^{\mathrm{b}^*}$, Venugopal V. Veeravalli$^\mathrm{c}$} 

$^\mathrm{a}$School of Data Science, The Chinese University of Hong Kong, Shenzhen \\

$^\mathrm{b}$Department of Electrical and Computer Engineering, The Ohio State University \\

$^\mathrm{c}$Department of Electrical and Computer Engineering, University of Illinois at Urbana-Champaign
\end{center}
}

\symbolfootnote[0]{\normalsize Address correspondence to Venugopal V. Veeravalli,
Department of Electrical and Computer Engineering, University of Illinois at Urbana-Champaign, Urbana, IL, 61801, USA; e-mail: vvv@illinois.edu.}
\symbolfootnote[0]{\normalsize $^*$Liyan Xie and Yuchen Liang have contributed equally to this work.}

{\small \noindent\textbf{Abstract:} 
The problem of quickest detection of a change in the
distribution of a sequence of independent observations is considered. It is assumed that the pre-change distribution is known (accurately estimated), while the only information about the post-change distribution is through a (small) set of labeled data. This post-change data is used in a data-driven minimax robust framework, where an uncertainty set for the post-change distribution is constructed using the Wasserstein distance from the empirical distribution of the data. The robust change detection problem is studied in an asymptotic setting where the mean time to false alarm goes to infinity, for which the least favorable post-change distribution within the uncertainty set is the one that minimizes the Kullback-Leibler divergence between the post- and the pre-change distributions. It is shown that the density corresponding to the least favorable distribution is an exponentially tilted version of the pre-change density and can be calculated efficiently. 
A Cumulative Sum (CuSum) test based on the least favorable distribution, which is referred to as the distributionally robust (DR) CuSum test, is then shown to be asymptotically robust. The results are extended to the case where the post-change uncertainty set is a finite union of multiple Wasserstein uncertainty sets, corresponding to multiple post-change scenarios, each with its own labeled data. The proposed method is validated using synthetic and real data examples.
}

\section{Introduction}\label{sec:intro}

Given sequential observations, the problem of quickest change detection (QCD) is to detect a potential change in their distribution that occurs at some change-point as quickly as possible, while not making too many false alarms \citep{siegmund1985sequential,basseville1993detection,poor2008quickest,tartakovsky2014sequential}. The QCD problem is of fundamental importance in mathematical statistics, and has seen a wide range of applications (see, e.g., \cite{veeravalli2013quickest, xie2021sequential} for overviews). 

In the classical formulation of the QCD problem \citep{page-biometrica-1954}, it is assumed that the observations are independent and identically distributed (i.i.d.) with known pre- and post-change distributions. In many applications of QCD, while it is reasonable to assume that the pre-change distribution is known (can be estimated accurately), the post-change distribution is rarely completely known. However, we may have access to a limited set of data corresponding to one or more possible post-change scenarios. For example, in the application of activity monitoring, historical data collected through wearable sensors can help provide prior information on people's activities and behaviors \citep{mukhopadhyay2014wearable}.

There has been a large body of work on the QCD problem when the pre- and/or post-change distributions have parametric uncertainty. The most prevalent approach to dealing with parametric uncertainty is the generalized likelihood ratio (GLR) approach, introduced in \cite{lorden1971} for the special case where the pre-change distribution is known and the post-change distribution has an unknown parameter.
The GLR approach for the QCD problem with general parametric distributions is studied in detail in \cite{lai-ieeetit-1998} and \cite{lai2010sequential}. An alternative approach to dealing with parametric uncertainty is the mixture-based approach, which was proposed and studied in \cite{pollakmixture}.

The QCD problem has also been studied in the non-parametric setting.
To detect a distributional change in this setting, one approach has been to replace the classical likelihood ratio with some other statistics and formulate the test. In \cite{xie2015mstat}, a test is proposed that compares the kernel maximum mean discrepancy (MMD) within a window to a given threshold. A way to set the threshold is also proposed that meets the false alarm rate asymptotically \citep{xie2015mstat}. Another approach has been to estimate the log-likelihood ratio through a pre-collected training set. This includes direct kernel estimation \citep{sugiyama2012direct} and neural network estimation \citep{moustakides2019training}. 
More recently, a non-parametric GLR test based on density estimation has been developed for the case where the post-change distribution is completely unknown without any pre-collected post-change training samples \citep{liang2023quickest}.

Another line of work for dealing with non-parametric distributional uncertainty is the one based on \textit{minimax} robust detection, in which it is assumed that the pre- and/or post-change distributions come from mutually exclusive uncertainty classes. This approach is of particular interest when distributional robustness is one of the objectives of the QCD formulation. Under certain conditions on the uncertainty classes, e.g., joint stochastic boundedness \citep{moulin-veeravalli-2018}, low-complexity solutions to the minimax robust QCD problem can be found \citep{Unnikrishnan2011}. Under more general conditions, in particular, weak stochastic boundedness, a solution that is asymptotically close to the minimax robust solution can be found, as the mean time to false alarm goes to infinity \citep{Molloy2017}. 

In the literature of robust hypothesis testing, a variety of uncertainty sets have been considered to address distributional uncertainties in a non-parametric way. One line of work is where the uncertainty set is constructed by selecting a nominal distribution as the center and choosing a deviation measure such that the set includes all distributions whose deviation from the nominal does not exceed a positive constant. Examples include the $\epsilon$-contamination model \citep{huber1965} and the KL-divergence uncertainty sets \citep{levy2008principles}. 
In the data-driven setting, the nominal distribution at the center is often chosen to be the empirical distribution of the observed samples. It should be pointed out that the KL divergence deviation measure may not be applicable when using the empirical distribution as the center of the set, since the KL divergence set will only include distributions that are absolutely continuous with respect to the empirical distribution. Other data-driven uncertainty sets that have been used for robust binary hypothesis testing problems include the Wasserstein uncertainty sets \citep{gao2018robust}, the kernel MMD uncertainty sets \citep{sun2021data}, and the Sinkhorn uncertainty sets \citep{wang2022data}. Another line of work is where the uncertainty set is constructed according to pre-specified constraints, such as moment constraints \citep{magesh2023robust}. 

In this paper, we consider a \emph{data-driven} minimax robust QCD problem, where the pre-change distribution is assumed to be {\it known}, and the only knowledge about the post-change distribution is through a limited set of data corresponding to one or more possible post-change scenarios.
For each possible post-change scenario, we define an empirical distribution that corresponds to the data collected under this scenario, and we construct the corresponding Wasserstein uncertainty set to contain all distributions such that their Wasserstein distance from the empirical distribution does not exceed some specified value (i.e., radius). We define the overall post-change uncertainty set in our minimax robust QCD formulation to be the union of the Wasserstein uncertainty sets across all possible post-change scenarios. Our goal is to find the asymptotically optimal robust detection procedure that minimizes the worst-case detection delay over the uncertainty set under false alarm constraints. We focus on the asymptotic setting where the mean time to false alarm goes to infinity.


Our contributions are summarized as follows.
\begin{enumerate}
    \item For the case where there is only one post-change scenario, we characterize the least favorable distribution (LFD) within the Wasserstein uncertainty set in closed-form. In particular, we show that the density corresponding to the LFD is an exponentially tilted form of the pre-change density. We therefore establish that the Cumulative Sum (CuSum) test based on the LFD, which we refer to as the distributionally robust (DR) CuSum test, is asymptotically minimax robust. We also characterize the minimum sufficient radius through concentration inequalities of empirical measures in Wasserstein distance.
    \item We extend the DR-CuSum test to construct an asymptotically robust solution for the case where the post-change uncertainty set is a union of multiple Wasserstein uncertainty sets, corresponding to multiple post-change scenarios, each with its own labeled data.
    \item  For an example with Gaussian observations, we compare the performance of the DR-CuSum test with that of the CuSum test that has knowledge of the Gaussian post-change model and uses the labeled data to produce a maximum likelihood estimate (MLE) of the post-change mean and variance. We show that with an appropriately chosen radius, the DR-CuSum test, which makes no distributional assumptions about the post-change observations, outperforms the Gaussian MLE CuSum test. We also demonstrate the performance of the DR-CuSum test using a real human activity detection dataset.
\end{enumerate}

The rest of the paper is organized as follows. 
In Section\,\ref{sec:formulation}, we define the asymptotically minimax robust QCD problem. In Section\,\ref{sec:single_pcs}, we introduce the Wasserstein uncertainty set and propose the DR-CuSum test for the case of a single post-change scenario. In Section\,\ref{sec:radius_select}, we discuss the choice of the radius of the uncertainty set and the corresponding guarantees of the resulting robust detectors. In Section\,\ref{sec:multiple_pcs}, we extend our results to the more general case with multiple post-change scenarios. In Section\,\ref{sec:num_res}, we provide some numerical examples using both synthetic data and a human activity detection dataset. Finally, in Section\,\ref{sec:conclusion}, we provide some concluding remarks and point to future research directions.

\section{Problem Setup} \label{sec:formulation}

Let $\{X_k,\ k\in \bN \}$ be a sequence of independent random vectors whose values are observed sequentially, with $\scrX$ denoting the observation space, i.e., $X_k \in \scrX$ for all $k\in \bN$. Let $\{\cF_k,\ k\in \bN \}$ be the filtration,
\[
\cF_k = \sigma(X_1, \ldots, X_k),
\]
with $\cF_0$ denoting the trivial sigma algebra. Let $P$ and $Q$ be probability measures on $\scrX$. 
At some unknown (yet deterministic) time $\nu$, the data-generating distribution changes from $Q$ to $P$,  i.e., 
\begin{equation}
\begin{array}{ll}
X_k  \stackrel{\text{iid}}{\sim} Q, &k = 1,2,\ldots,\nu-1, \\
X_k \stackrel{\text{iid}}{\sim} P, &k = \nu,\nu+1,\ldots
\end{array} 
\label{eq:hypothesis}
\end{equation}
We assume that the pre-change measure $Q$ is {\it known}, while only partial knowledge of the post-change measure $P$ is available through a set of labeled (training) data.

Let $\mathbb P_\nu^{P}$ denote the probability measure on the observation sequence when the change-point is $\nu$, and the pre- and post-change measures are $Q$ and $P$, respectively, and let $\mathbb E_\nu^{P}$ denote the corresponding expectation. Also, $\mathbb P_\infty$ and $\mathbb E_\infty$ denote the probability and expectation operator when there is no change (i.e., $\nu=\infty$). For brevity, we write $\mathbb P^P$ and $\mathbb E^P$ as the probability and expectation operator when all the samples are generated from $P$ (i.e., when $\nu=1$).

The goal in QCD is to raise an alarm after the unknown change-point $\nu$ as quickly as possible, while keeping the false alarm rate below a pre-specified level. The detection is performed through a {\em stopping time} $\tau$ on the observation sequence at which the change is declared \citep{xie2021sequential}. 


\subsection{QCD Optimization Problem and CuSum Test}

\noindent \textsl{False Alarm Measure.} We measure the false alarm performance of a QCD test (stopping time) $\tau$ in terms of its mean time to false alarm $\bE_\infty\left[\tau\right]$, and we denote by $\cC(\gamma)$  the set of all tests for which the mean time to false alarm is at least $\gamma$, i.e., 
\begin{align}\label{eq:Cgamma_def}
\cC(\gamma)= \left\{ \tau: \; \bE_\infty\left[\tau\right] \geq \gamma \right\}.
\end{align}

\noindent \textsl{Delay Measure.} We use the commonly used worst-case measure for delay due to Lorden \citep{lorden1971}. Specifically, for post-change distribution $P$ and test $\tau$, we set\footnote{Alternatively we may also consider the Pollak's measure \citep{poll-astat-1985} $\CADD^P(\tau) = \operatorname{\sup}_{{\nu \geq 1}}\ \mathbb E^P_\nu [\tau-\nu| \tau\geq \nu]$.}
\begin{equation}\label{eq:WADDdef}
\WADD^P(\tau) =  \underset{\nu \geq 1}{\operatorname{\sup}} \esssup \ \mathbb E^P_\nu\left[(\tau-\nu+1)^+| \mathcal F_{\nu-1}\right].
\end{equation}


\noindent \textsl{QCD Optimization Problem.} When both $Q$
and $P$ are known \emph{a priori}, the optimization problem of interest is
\begin{equation}\label{eq:orig_qcd}
\inf_{\tau \in \cC(\gamma)} \WADD^{P}(\tau).
\end{equation}
Lorden showed that Page's Cumulative Sum (CuSum) test \citep{page-biometrica-1954} solves the problem in \eqref{eq:orig_qcd} asymptotically; this result was strengthened to exact optimality in \cite{moustakides1986optimal}. The stopping time of the CuSum test is given by
\begin{align}\label{eq:cusum_test}
\tau_b = \inf\left\{k \in \bN:  S_k  \geq b\right\},
\end{align}
with the CuSum statistic given by the recursion (for $k \geq 1$):
\begin{align}
\label{eq:cusum_stat}
    S_k = \max_{1 \leq j \leq k} \sum_{\ell=j}^k \log \frac{p(X_\ell)}{q(X_\ell)} = \left(S_{k-1}\right)^{+} + \log \frac{p(X_k)}{q(X_k)}, \quad S_0 = 0,
\end{align}
 and $b$ is chosen to meet the false alarm constraint of $\gamma$.  Here $p$, $q$ are the respective probability density functions (pdfs) of the measures $P$ and $Q$ with respect to some common dominating measure.

\subsection{Asymptotically Minimax Robust QCD}
As mentioned previously, we have limited knowledge about the post-change distribution $P$. One way to deal with this distributional uncertainty is to assume that  $P \in \calP$, where $\calP$ is a family of probability measures representing potential post-change distributions.
In the minimax robust QCD formulation, the goal is to solve the following optimization problem,
\begin{equation}\label{eq:Lorden}
\inf_{\tau \in C(\gamma)} \sup_{P\in\calP}\WADD^{P}(\tau),
\end{equation}
where $C(\gamma)$ is as defined in \eqref{eq:Cgamma_def}. 
As is standard practice in the analysis of QCD procedures, we are primarily interested in the asymptotically optimal solution to \eqref{eq:Lorden} as $\gamma \to \infty$. 
A solution $\tau^*\in C(\gamma)$ is called {\it first-order asymptotically minimax robust} for \eqref{eq:Lorden} if 
\[  \sup_{P\in\calP}\WADD^{P}(\tau^*) = \inf_{\tau\in C(\gamma)} \sup_{P\in\calP}\WADD^{P}(\tau) \cdot (1+o(1)),\]
where, as throughout this paper, $o(1) \to 0$ as $\gamma \to \infty$ \citep{Molloy2017}.

Determining the asymptotically minimax robust solution to robust QCD problems is facilitated by the following weak stochastic boundedness (WSB) condition on the uncertainty classes in the pre- and post-change regimes.

\begin{definition}[Weak Stochastic Boundedness \citep{Molloy2017}]\label{def:weak}
Let $\calP_0$ and $\calP_1$ be sets of distributions on a common measurable space where $\calP_0 \cap \calP_1 = \emptyset$. The pair $(\calP_0,\calP_1)$ is said to be weakly stochastically bounded by the pair of distributions $(P_0^\ast, P_1^\ast)$ if
\begin{equation*}\label{eq:cond1}
  \mathsf{KL}(P_1^\ast ||P_0^\ast) \leq \mathsf{KL}(P_1 ||P_0^\ast) - \mathsf{KL}(P_1 || P_1^\ast), \ \forall P_1 \in \calP_1,  
\end{equation*}
and
\begin{equation*}\label{eq:cond2}
\mathbb E^{P_0}\left[\frac{dP_1^\ast}{dP_0^\ast}(X)\right]\leq \mathbb E^{P_0^\ast}\left[\frac{dP_1^\ast}{dP_0^\ast}(X)\right]=1,\ \forall P_0 \in \calP_0,
\end{equation*}
where
\[
\mathsf{KL}(P||Q) = \int_\cX  \log \bigg(\frac{dP}{dQ}(x)\bigg)dP(x),
\]
and $\frac{dP}{dQ}(x)$ is the Radon-Nikodym derivative of $P$ with respect to $Q$ with $\frac{dP}{dQ}(x)=\infty$ when the derivative does not exist. The pair of distributions $(P_0^\ast, P_1^\ast)$ are called least favorable distributions (LFDs).
\end{definition}

Intuitively, the LFDs can be viewed as a representative pair of distributions within the uncertainty sets on which the stopping time reaches the worst-case performance. In this paper, we assume that the pre-change distribution is known, and thus there is no pre-change uncertainty. This corresponds to the special case that $\calP_0$ is a singleton. The following lemma follows directly from \cite[Prop. 1 (iii)]{Molloy2017}.

\begin{lemma}\label{lemma:wsb}
For the singleton set $\calQ = \{Q\}$ and a \emph{convex} set of distributions $\calP$ where $Q \notin \calP$, $(\calQ,\calP)$ is weakly stochastically bounded by the pair of distributions $(Q, P^\ast)$, where
\begin{equation}\label{eq:kl_min_general}
P^\ast = \arg \min_{P \in \calP} \mathsf{KL}(P||Q).
\end{equation}
\end{lemma}

Let $p^*,q$ be pdfs of $P^*$ and $Q$, respectively,  with respect to a common dominating measure. 
Then applying \citep[Theorem 3]{Molloy2017}, we conclude that the first-order asymptotically minimax robust solution to \eqref{eq:Lorden}, as $\gamma \to \infty$, is given by the CuSum test with pre-change pdf $q$ and post-change pdf $p^\ast$, i.e.,
\begin{align}\label{eq:cusum_test_opt}
\tau_{\mathrm{DR}}
= \inf\left\{k \in \bN:  S_k  \geq b\right\},
\end{align}
with $S_k$ satisfying the recursion (for $k\geq 1$):
\begin{align}
\label{eq:cusum_stat_opt}
    S_k = \left(S_{k-1}\right)^{+} + \log \frac{p^\ast (X_k)}{q(X_k)}, \quad S_0= 0,
\end{align}
 and $b$ chosen to meet the false alarm constraint of $\gamma$.

\section{Asymptotically Minimax Robust QCD under Wasserstein Uncertainty Sets: Single Post-Change Uncertainty Set} 
\label{sec:single_pcs}


As described in Section~\ref{sec:intro}, we are interested in a \emph{data-driven} version of the minimax robust QCD problem, where the only knowledge about the post-change distribution is through a limited set of labeled data corresponding to one or more potential post-change scenarios. We use this data to construct an uncertainty set for the post-change distribution. We begin by considering the simplest case where there is only {\it one} possible post-change scenario. We define a \emph{Wasserstein} uncertainty set as the set of all distributions whose Wasserstein distance from the empirical distribution corresponding to the data from the post-change scenario is less than some specified value, called the \emph{radius}. 
Our main finding is that under the Wasserstein uncertainty set, the density of the post-change LFD, i.e., the solution to \eqref{eq:kl_min_general}, is an exponentially tilted version of the pre-change density. 

\subsection{The Wasserstein Uncertainty Model}

Suppose we have $n$ training data $\{\omega_1,\ldots,\omega_n\}$ that are independently sampled from the post-change regime, then we choose the nominal distribution of the uncertainty set to be the empirical distribution of those historical samples, i.e., $\hatP_n=\frac1{n}\sum_{i=1}^{n} \delta_{\omega_{i}}$, where $\delta_{\omega}$ corresponds to the Dirac measure at $\omega$.

We construct the uncertainty set $\calP_n$ in a data-driven manner using Wasserstein distance of order $s>0$ as follows. The uncertainty set $\calP_n$ is defined as the set of all probability measures that are close to $\hatP_n$ with respect to the Wasserstein distance measure $\wass_s(\cdot,\cdot)$,
\begin{equation}\label{eq:uncertainty}
\begin{aligned}
\calP_n &= \{P\in \scrP_s: \wass_s(P,\hatP_n) \leq r_s\}, 
\end{aligned}
\end{equation}
where $\scrP_s$ is the set of all Borel probability measures $P$ on the sample space $\scrX$ such that the moment bound $\int_\scrX c^s(x,x_0) d P(x) < \infty$ holds for all $x_0\in \scrX$, where $c(\cdot,\cdot): \scrX \times \scrX \rightarrow \mathbb R_+$ is a metric and $r_s\geq 0$ is the radius parameter controlling the size of the uncertainty set. If $r_s$ is set to zero, then the uncertainty set is simply a singleton that only contains the empirical measure $\hatP_n$. 
Here, for any two probability measures $P,\Tilde{P}\in \scrP_s$, their Wasserstein distance (of order $s$) with metric $c(\cdot,\cdot)$ equals to 
\begin{equation}\label{eq:was}
\wass_s(P,\Tilde{P}) :=\cbrc{\min_{\Gamma\in\Pi(P,\Tilde{P})} \int_{\scrX\times \scrX} c^s(\omega,\Tilde{\omega}) d\Gamma (\omega, \Tilde{\omega}) }^{1/s}, 
\end{equation}
where $\Pi(P,\Tilde{P})$ is the set of all joint probability measures on $\scrX \times \scrX$ with marginal distributions $P$ and $\Tilde{P}$, respectively \citep{villani2003topics}. 
In this paper, we restrict $s\geq 1$ and mainly use $s=2$ in numerical experiments. 
In the subsection following, we omit the dependency on the sample size $n$ and write the empirical distribution and uncertainty set as $\hatP$ and $\calP$, respectively,  for brevity.

\subsection{Least Favorable Distribution: Exponential Tilting}

To obtain the LFD, the goal is to solve the following optimization problem,
\begin{equation}\label{eq:kl_min}
 \min_P \mathsf{KL}(P||Q),~\text{such that } \wass_s(P,\hatP) \leq r_s.   
\end{equation}
Below, we show that the optimal solution to \eqref{eq:kl_min} can be found as an exponentially tilted version of the pre-change distribution. In the following, we assume that the pre-change distribution $Q$ has pdf $q$ with respect to the Lebesgue measure $\mu$. 

\begin{theorem}\label{thm:lfd}
The pdf of the least favorable distribution to the problem \eqref{eq:kl_min}, with respect to the dominating Lebesgue measure $\mu$, satisfies $p^*(x) \propto q(x) e^{-C_{\lambda^*,u^*}(x)}$, where $q(x)$ is the pre-change pdf with respect to $\mu$, and $\lambda^*\geq 0,u^*\in\mathbb{R}^n$ are the optimizers of the following problem,
\begin{equation}\label{eq:dual_max}
\max_{\substack{\lambda\geq 0, u \in \mathbb R^{n}}} \bigg\{ -\lambda r_s^s +  \frac{1}{n}\sum_{i=1}^{n} u_i  -\log \eta(\lambda,u) \bigg\},   
\end{equation}
with
\begin{equation}
    \label{eq:C}
    C_{\lambda,u}(x):= \min_{1\leq i\leq n} \{\lambda c^s(x,\omega_i) - u_i\},
\end{equation}
and
\[
\eta(\lambda,u):=\int q(x) e^{-C_{\lambda,u}(x)}d\mu(x).
\]
\end{theorem}

\begin{proof}
   We first consider $s=1$, with the radius being denoted by $r_1$. Denote by $\Pi(P,\hatP)$  the space of all joint distributions on $\scrX\times \scrX$. Note that the empirical distribution $\hatP$ is {\it discrete} with finite support $\{\omega_1,\ldots,\omega_n\}$. 
   Without loss of generality, we assume that $P^\ast$, the optimal solution to \eqref{eq:kl_min}, is {\it absolutely continuous} with respect to the pre-change measure $Q$ because otherwise $\mathsf{KL}(P^\ast||Q)=\infty$. Since $Q$ is dominated by $\mu$, $P^\ast$ is also dominated by $\mu$.
   
   Therefore, we consider all joint distributions $\Pi(P,\hatP)$ with a continuous marginal $P$ (with respect to $\mu$) and a discrete marginal $\hatP$. Their joint distribution $\Pi(P,\hatP)$ can be characterized by the mixed joint density, denoted as
   \begin{equation}\label{eq:def_joint}
      \pi(x,\omega_i)=\frac{1}{n}f_i(x), \ \text{where } f_i(x)\geq 0, \ \int_{\scrX}f_i d\mu(x)=1, \ \forall i=1,2,\ldots,n.  
   \end{equation}
   Here the term $1/n$ corresponds to the probability mass function of its second marginal, while $f_i(x)$ can be viewed as the conditional density function (with respect to the same dominating measure $\mu$) of the first variable given that the second variable equals $\omega_i$. Thus $\sum_{i=1}^n \pi(x,\omega_i)=\frac{1}{n}\sum_{i=1}^nf_i(x)$ is the probability density function (with respect to $\mu$) of its first marginal $P$, which we denote as $p$. 
   Also, define ${\cal P}_{\mu}$ to be the set of all distributions absolutely continuous with respect to measure $\mu$. Note that ${\cal P}_{\mu}$ is trivially convex.

   To solve the constrained minimization problem in \eqref{eq:kl_min}, we note that $\mathsf{KL}(P||Q)$ is a convex functional in $P$, and $\wass_1(P,\hatP ) - r_1$ 
   is a convex mapping of $P$ into $\mathbb{R}$. Since $\calP_\mu$ is dense with respect to the Wasserstein distance for Lebesgue measure $\mu$, meaning that there exists a distribution in $\calP_\mu$ that is arbitrarily close to any given measure $P$. Thus it is easy to find some $P^\circ_1 \ll \mu$ close to the empirical samples such that $\wass_1(P^\circ_1,\hatP) \leq r_1$. 
   Also, $\inf\cbrc{\mathsf{KL}(P||Q): P \in {\cal P}_{\mu}, \wass_1(P,\hatP ) \leq r_1} \geq 0 > -\infty$. Then by Lagrange duality \cite[Sec~8.6~Thm~1]{luenberger1969optimization} we have
    \begin{equation}\label{eq:dual_1}
     \inf\cbrc{\mathsf{KL}(P||Q): P \in {\cal P}_{\mu}, \wass_1(P,\hatP ) \leq r_1} = \max_{\lambda \geq 0} \inf_{P \in {\cal P}_{\mu}} \brc{ \mathsf{KL}(P||Q) + \lambda \wass_1(P,\hatP ) - \lambda r_1 },    
    \end{equation}
    and this maximum on the right-hand side is achieved at some $\lambda^* \geq 0$.

    
    Using the definition of Wasserstein distance, we have 
    \[
    \wass_1(P,\hatP ) = \inf_{\pi\in\Pi(P,\hatP)} \sum_{i=1}^n\int_\scrX c(x,\omega_i)\pi(x,\omega_i)d\mu(x),
    \]
    which, after substituting into \eqref{eq:dual_1}, results in the following dual optimization problem to \eqref{eq:kl_min},   
    \[ 
    \max_{\lambda \geq 0} \brc{\! - \lambda r_1 + \inf_{\pi} \!\! \int_\scrX \!\Big(\sum_{i=1}^n\pi(x,\omega_i)\Big)\log\frac{\sum_{i=1}^n\pi(x,\omega_i)}{q(x)} d\mu(x) +  \sum_{i=1}^n \int_\scrX \lambda c(x,\omega_i) \pi(x,\omega_i) d\mu(x)  \!}.
    \]
    Note that the inner problem can be written as 
    \[
    \begin{aligned}
     \inf_{\pi} & \brc{\int_\scrX \Big(\sum_{i=1}^n\pi(x,\omega_i)\Big)\log\frac{\sum_{i=1}^n\pi(x,\omega_i)}{q(x)} d\mu(x) +  \sum_{i=1}^n \int_\scrX \lambda c(x,\omega_i) \pi(x,\omega_i) d\mu(x) } \\
     {s.t}\quad &  \int_{\scrX} \pi(x,\omega_i)d\mu(x)=\frac1n, \forall i=1,2,\ldots,n,
     \ \sum_{i=1}^n\int_{\scrX} \pi(x,\omega_i)d\mu(x)=1. 
    \end{aligned}
    \]
    By the definition of mixed joint density $\pi(x,\omega_i)$ in \eqref{eq:def_joint}, the above problem is equivalent to the following optimization problem over non-negative functions $f_1,f_2,\ldots,f_n$, 
    \[
    \begin{aligned}
     \inf_{f_1,\ldots,f_n } & \brc{ \frac{1}{n}\int_\scrX \Big(\sum_{i=1}^nf_i(x)\Big)\log\frac{\frac{1}{n}\sum_{i=1}^nf_i(x)}{q(x)} d\mu(x) +  \frac{1}{n}\sum_{i=1}^n\int_\scrX \lambda c(x,\omega_i) f_i(x)d\mu(x)  }  \\
     {s.t}\quad &  \int_{\scrX} f_i(x)d\mu(x)=1, \forall i=1,2,\ldots,n, \ \frac1n \sum_{i=1}^n\int_{\scrX} f_i(x)d\mu(x)=1. 
    \end{aligned}
    \]

    We now introduce Lagrangian multipliers $u_i\in\mathbb{R},i=1,2,\ldots,n$ and $\Check{\eta}\in\mathbb{R}$ for the constraints. By strong duality \citep{luenberger1969optimization}, we have that the value of the above minimization problem becomes  
    \[
    \max_{u_1,\ldots,u_n} \inf_{f_1,\ldots,f_n} \brc{ \frac1n\int_\scrX \sum_{i=1}^nf_i(x)\Big(\log\frac{\sum_{i=1}^n \frac1n f_i(x)}{q(x)}   + \lambda c(x,\omega_i) - u_i + \Check{\eta} \Big) d\mu(x)  + \frac1n\sum_{i=1}^n u_i - \Check{\eta}} . 
    \]
    We can then solve the inner infimum problem for each value $x$ to get the optimal $f_1^*(x),\ldots,f_n^*(x)$, since the function inside the integral only depends on the particular value of $x$. 
    From Lemma\,\ref{lem:opt_lem} in the appendix, we have that for each $x$, the inner minimization problem over $f_1(x),\ldots,f_n(x)$ has optimal solution satisfying 
    \[
    \frac1n\sum_{i=1}^n f_i^*(x) = q(x) e^{-\min_i (\lambda c(x,\omega_i)-u_i) - \Check{\eta} - 1} = q(x) e^{-C_{\lambda,u}(x) - \Check{\eta} - 1},
    \]
    where the last equality is due to the definition in \eqref{eq:C}, and the corresponding optimum value for each $x$ is $-q(x) e^{-C_{\lambda,u}(x) - \Check{\eta} - 1}$. Moreover, to satisfy the constraint, the optimal Lagrangian multiplier $\Check{\eta}$ must satisfy 
    \[
    \Check{\eta} + 1 = \log \brc{ \int_{\scrX} 
    q(x) e^{-C_{\lambda,u}(x)} d\mu(x)}.
    \]
    Therefore, $ \frac1n\sum_{i=1}^n f_i^*(x)$ is a probability density function 
    and the corresponding objective value satisfies
    \[
    \begin{aligned}
     & \frac1n\int_\scrX \sum_{i=1}^nf_i^*(x)\Big(\log\frac{\sum_{i=1}^n \frac1n f_i^*(x)}{q(x)}   + \lambda c(x,\omega_i) - u_i + \Check{\eta} \Big) d\mu(x)  - \Check{\eta}\\
     & = -\int_\scrX q(x) e^{-C_{\lambda,u}(x) - \Check{\eta} - 1}d\mu(x)  - \Check{\eta} =  - 1 - \Check{\eta}\\
     &  = - \log \brc{ \int_{\scrX} 
    q(x) e^{-C_{\lambda,u}(x)} d\mu(x)} =: -\log \eta(\lambda,u),
    \end{aligned}
    \]
    where for notational simplicity we have defined $\eta(\lambda,u) := \int  q(x) e^{-C_{\lambda,u}(x)}d\mu(x)$.
    The resulting outer maximization problem is as in \eqref{eq:dual_max}. After solving the dual optimization problem \eqref{eq:dual_max} and obtaining the optimal dual variable $\lambda^*,u^*$, we arrive at the optimal solution to the problem in \eqref{eq:kl_min}, which is $p^\ast = p^{\lambda^*,u^*}(x) =  \frac1n\sum_{i=1}^n f_i^*(x) \propto q(x) e^{-C_{\lambda^*,u^*}(x)}$, or more specifically 
    $$ 
    p^\ast(x) = q(x)e^{-C_{\lambda^*,u^*}(x)-\eta(\lambda^*,u^*)}.
    $$
    In the case of a general order $s \geq 1$, we will have $c^s(x,\omega_i)$ in the above arguments, and the proof follows similarly. The resulting LFD pdf $p^*(x)$ is still an exponentially tilting of $q(x)$.
\end{proof}

\begin{remark}[Convexity]
The objective function in \eqref{eq:dual_max} is concave, and therefore the optimization problem is convex and can be solved efficiently by off-the-shelf solvers. Indeed, for any $x$, $C_{\lambda,u}(x)$ is concave in $(\lambda,u)$,  since for any $(\lambda,u)$ and $(\lambda',u')$ and any $\theta\in(0,1)$, we have $C_{\theta\lambda+(1-\theta)\lambda',\theta u + (1-
\theta)u'}(x)\geq \theta C_{\lambda,u} + (1-\theta) C_{\lambda',u'}$.
Thus, $q(x)e^{-C_{\lambda,u}(x)}$ is log-convex in $(\lambda,u)$; from \cite[p. 106]{boyd2004convex}, we have that $\int q(x)e^{-C_{\lambda,u}(x)}d\mu(x)$ is also log-convex in $(\lambda,u)$, and hence $\log\eta(\lambda,u)$ is convex.
\end{remark}

After we obtain the pdf $p^*$ of the least favorable distribution, we can construct the log-likelihood ratio at sample $X_k$ as:
\begin{equation}\label{eq:dr-cusum-stat}
\log \frac{p^\ast(X_k)}{q(X_k)}  = -C_{\lambda^*,u^*}(X_k)-\log\eta(\lambda^*,u^*),    
\end{equation}
which, after substitution into \eqref{eq:cusum_stat_opt}, gives the \textit{Distributionally Robust CuSum} (\textit{DR-CuSum}) statistics under Wasserstein uncertainty sets. And the corresponding stopping time in \eqref{eq:cusum_test_opt}, the \textit{DR-CuSum test}, denoted by $\tau_{\mathrm{DR}}$, is the first-order asymptotically minimax robust solution to problem \eqref{eq:Lorden}.

Note that the log-likelihood ratio \eqref{eq:dr-cusum-stat} depends on the optimal dual variables $\lambda^*$, $u^*$, and all the available post-change empirical samples $\{\omega_1,\ldots,\omega_n\}$ through the function $C_{\lambda^*,u^*}(x)$. 
By exploiting the convexity property, the optimal dual variables $\lambda^*$ and $u^*$ can be pre-computed from \eqref{eq:dual_max}. Thus the second term $\eta(\lambda^*,u^*)$ in \eqref{eq:dr-cusum-stat} can also be pre-computed. 
Finally, the term $C_{\lambda^*,u^*}(x)$ can be computed easily, given the pre-computed optimizers $\lambda^*$ and $u^*$. Therefore, the DR-CuSum statistics can be easily updated online, resulting in an efficient test for detecting the change. Furthermore, in some special situations, it may be possible to calculate the log-likelihood ratio \eqref{eq:dr-cusum-stat} in closed-form as detailed in the following example.
\begin{example}
 For illustrative purposes, we study the LFD under the setting where $c(x,x') = \norm{x-x'}_2$ and the Wasserstein order $s=2$. We also assume univariate data and the standard normal pre-change distribution, i.e., $Q = \mathcal N(0,1)$, and the dominating measure $\mu$ is the Lebesgue measure on $\mathbb R$. We first derive a closed-form solution of LFD for the extreme case where the number of empirical samples $n=1$.
 In this case, the function $C_{\lambda,u}(x)$ defined in Theorem\,\ref{thm:lfd} equals $C_{\lambda,u}(x) = \lambda (x-\omega_1)^2-u,~\forall x$. Then,
\begin{align*}
    \eta(\lambda,u)=& \int \frac{1}{\sqrt{2 \pi}} e^{-x^2/2 - \lambda x^2 + 2 \lambda \omega_1 x - \lambda \omega_1^2 + u } d x
    =\frac{1}{\sqrt{1 + 2 \lambda}} e^{-\frac{\lambda}{1 + 2 \lambda} \omega_1^2 + u },
\end{align*}
and the optimal solution to problem  \eqref{eq:dual_max} is given by
\[ 
\lambda^\ast = \frac{\omega_1^2}{\sqrt{1 + 4 r \omega_1^2} - 1} - \frac{1}{2},\ ~\text{if}~r_2 \leq 1 + \omega_1^2, 
\]
and $\lambda^\ast = 0$ otherwise. 
Note that a large radius yields $\lambda^\ast = 0$ and the LFD will thus be identical to the pre-change distribution. In practice, the radius has to be carefully chosen to avoid such scenarios so that the robust detection problem is well-defined.

For general $n > 1$, we provide an efficient LFD-solving algorithm based on the following decomposition
\[ 
\eta(\lambda,u) = \sum_{i=1}^n \int_{I_i} q(x) e^{-\lambda c^s(x,\omega_i) + u_i} d x, 
\]
where $I_i := \cbrc{x \in \mathbb{R}: \lambda c^s(x,\omega_i) - u_i \leq \lambda c^s(x,\omega_j) - u_j,~\forall j \neq i}$. Under previous conditions that $s=2$ and $c(x,x') = \norm{x-x'}_2$, for $i = 1, \dots, n$, we have that
\[ 
\lambda (x - \omega_i)^2 - u_i \leq \lambda (x - \omega_j)^2 - u_j 
\]
which is equivalent to 
\[
\begin{array}{ll}
   2(\omega_j-\omega_i) x \leq \frac{u_i-u_j}{\lambda} + \omega_j^2 - \omega_i^2,~\forall j \neq i & \text{if}~\lambda > 0 \\
   u_i \geq u_j  & \text{if}~\lambda = 0 
\end{array} 
\]
This implies that $I_i$ is a connected interval, i.e. $I_i = [\underline{l}_i, \bar{l}_i]$. When $\lambda > 0$, we have
\begin{align*}
    \underline{l}_i &= \max_{j: \omega_j < \omega_i}  \cbrc{\frac{u_i-u_j}{2 \lambda(\omega_j-\omega_i)} + \frac{\omega_j + \omega_i}{2}},\\
    \bar{l}_i &= \min_{j: \omega_j > \omega_i} \cbrc{\frac{u_i-u_j}{2 \lambda(\omega_j-\omega_i)} + \frac{\omega_j + \omega_i}{2}},
\end{align*}
and the decomposition yields
\[ 
\eta(\lambda,u) = \sum_{i=1}^n \ind\cbrc{\underline{l}_i < \bar{l}_i } \frac{\exp\brc{u_i-\frac{\lambda \omega_i^2}{1 + 2 \lambda}}}{2 \sqrt{2 \lambda + 1}} \brc{\erf\brc{\frac{2 \lambda (\bar{l}_i - \omega_i) + \bar{l}_i}{\sqrt{4 \lambda + 2}}} - \erf\brc{\frac{2 \lambda (\underline{l}_i - \omega_i) + \underline{l}_i}{\sqrt{4 \lambda + 2}}} }, \]
where $\erf(z):=\frac{2}{\sqrt{\pi}}\int_0^z e^{-t^2}dt$ is the error function. 
When $\lambda = 0$, $I_i = \mathbb{R}$ if $i = \argmax u_i$ and $I_i = \emptyset$ otherwise, and thus $\eta(0,u) = \exp\brc{\max_i u_i}$. 

For general pre-change distributions, it is not guaranteed that we can find analytical solutions for the LFD for $n> 1$. However, we note that the solution to the optimization problem in \eqref{eq:dual_max} is easy to compute numerically for any $n$, $s$, and $r_s$, regardless of the type of the pre-change distribution. This is due to the convexity of the problem in \eqref{eq:dual_max}.

\end{example}

\begin{remark}[Unknown Pre-change Distribution] \label{rmk:emp_pre}
In some applications, the pre-change distribution may also need to be estimated from historical data $\xi_1,\ldots,\xi_{N}$ in the pre-change regime, with $N$ generally being much larger than $n$. The optimization problem \eqref{eq:dual_max} is easily adaptable to such case, because the integral $\eta(\lambda,u) = \int q(x) e^{-C_{\lambda,u}(x)}d\mu(x) $ can be  approximated directly by sample average $\frac1N\sum_{i=1}^N e^{-C_{\lambda,u}(\xi_i)}$, without have to estimate the pre-change density. After obtaining $\eta(\lambda,u)$, we can then update the DR-CuSum statistics by accumulating the log-likelihood ratio in \eqref{eq:dr-cusum-stat}, without knowing the exact pre-change distribution function. Note, however, that the resulting test is \textit{not} necessarily the minimax robust solution to the problem \eqref{eq:Lorden} due to the empirical approximation in computing $\eta(\lambda,u)$.  
\end{remark}

\section{Choice of Radius and the Resulting Detection Delay}
\label{sec:radius_select}

In this section, we discuss the choice of the radius parameter $r_s$ in constructing the uncertainty set $\calP_n$ in \eqref{eq:uncertainty}, which is an essential parameter to balance the robustness and effectiveness of the DR-CuSum test. If the radius $r_s$ is too small, the uncertainty set might be far away from the true post-change data distributions, thus making the solution not minimax robust. On the other hand, if the radius is too large, the solution would be too conservative and would lead to large detection delays in the presence of a change. Therefore, it is important to select a proper radius when performing such data-driven robust detection.
We adopt the commonly used principle from distributionally robust optimization (DRO) of choosing the radius such that the true data-generating distribution is included within the uncertainty set with high probability \citep{kuhn2019wasserstein,mohajerin2018data}.

We consider the ideal case where the empirical samples $\{\omega_1,\ldots,\omega_n\}$ are generated from the true post-change distribution $P$. 
The idea is to guarantee that the Wasserstein ball of the post-change samples contains the true post-change distribution but {\it not} the pre-change distribution. We first list a known concentration of the empirical measure under Wasserstein distance under the $T_s(c)$ inequality condition detailed below. Examples of $T_s(c)$ inequalities can be found in Appendix \ref{sec:app_T1T2}.

\begin{definition}[$T_s(c)$ inequality \citep{raginsky2013concentration}]\label{def:T_ineq}
We say that a probability measure $P$ satisfies an $L^s$ transportation-cost inequality with constant $c>0$ (which is referred to as the $T_s(c)$ inequality), if for every probability measure $Q\ll P$ we have 
\[
\wass_s(P,Q) \leq \sqrt{2c\mathsf{KL}(Q||P)}.
\]    
\end{definition}

\begin{theorem}[Empirical Concentration \citep{bolley2007quantitative}]
\label{thm:wass_conc}
Let $s\in[1,2]$ and let $P$ be a probability measure on $\mathbb R^d$ satisfying a $T_s(c)$ inequality, then for any $d'>d$ and $c'>c$, there exists some constants $N_0$, depending only on $c'$, $d'$ and some square-exponential moment of $P$, such that for any $\epsilon>0$ and $n \geq N_0\max(\epsilon^{-(d'+2)},1)$, 
\[
\mathbb P^P\{ \wass_s(P,\hatP_n) > \epsilon\} \leq e^{-\gamma_s n \epsilon^2/2c'},
\]
where $\gamma_s = 1$ if $s\in[1,2)$ and $\gamma_s = 3-2\sqrt{2}$ if $s=2$. Here recall $\mathbb P^P$ is the probability measure on the samples that are distributed i.i.d. as $P$. 
\end{theorem}


We first give an {\it upper bound} requirement for the radius which guarantees that the pre-change distribution is excluded from the post-change uncertainty set with high probability. Since the pre-change distribution $Q$ is known, given any empirical measure $\hatP_n$, we can calculate
$\wass_s(Q,\hatP_n)$ and select the radius $r_s$ such that 
\begin{equation}\label{eq:radius_upper}
   r_s < \wass_s(Q,\hatP_n). 
\end{equation}
This choice guarantees that $Q \notin \calP_n$, i.e., that the pre-change distribution is excluded from the post-change uncertainty set, thus making the detection problem valid. However, it is worthwhile emphasizing that for theoretical considerations below, we assume that $\calP_n$ and the resulting LFD are essentially random due to the randomness of empirical samples. The corollary below calculates an {\it upper bound} for the radius considering the randomness of empirical samples.

\begin{corollary}[Upper Bound for Radius] \label{cor:radius_upper}
Fix $\delta\in(0,1)$ and $s \in [1,2]$. Suppose that the pre- and post-change distributions $Q$, $P$ are probability measures on $\mathbb R^d$, and that $P$ satisfies the $T_s(c)$ inequality.
Suppose that $n\geq N_0\max(r_s^{-(d+2)},1)$ where $N_0$ is the same as in Theorem~\ref{thm:wass_conc}. Then, if we set the radius as
\[ r_s \leq \overline{r}_{\delta,n} := \wass_s(P,Q) - \sqrt{\frac{2|\log\delta|c}{\gamma_s n}},\]
it is guaranteed that with probability at least $1-\delta$ we have $Q \notin \calP_n$.
\end{corollary}

\begin{proof}
From the triangle inequality satisfied by the Wasserstein distance, we have
\[ \wass_s(Q,\hatP_n) \geq \wass_s(P,Q) - \wass_s(P,\hatP_n), \]
and thus
\begin{align*}
    \mathbb P^P\{\wass_s(Q,\hatP_n) < r_s\} &\leq \mathbb P^P\{\wass_s(P,\hatP_n) > \wass_s(P,Q) - r_s\}\\
    &\leq \exp\brc{-\gamma_s n (\wass_s(P,Q) - r_s)^2/2c},
\end{align*}
where the last inequality follows from Theorem~\ref{thm:wass_conc}. Now, if
\[  r_s \leq \wass_s(P,Q) - \sqrt{\frac{2|\log\delta|c}{\gamma_s n}}, \]
then $\mathbb P^P\{Q\in \calP_n\}=\mathbb P^P\{\wass_s(Q,\hatP_n) < r_s\} \leq \delta$.
\end{proof}

The above upper bound can be calculated directly when we have prior knowledge of $\wass_s(P,Q)$. Furthermore, this corollary guarantees the appropriateness of any radius less than the true Wasserstein distance $\wass_s(P,Q)$ when $n$ is sufficiently large. 
When we lack knowledge of $\wass_s(P,Q)$, as might be the case in practice, the upper bound $\overline{r}_{\delta,n}$ is only of theoretical interest, and we can use Equation \eqref{eq:radius_upper} to determine an appropriate radius.

Next, we present a {\it lower bound} for $r_s$ to guarantee that $P \in \calP_n$ with high probability. Furthermore, we also characterize the delay performance when such a condition is satisfied. The corollary below is a direct application of Theorem~\ref{thm:wass_conc}.

\begin{corollary}[Lower Bound for Radius] \label{cor:radius_lower}
Under the same conditions as in Corollary~\ref{cor:radius_upper}, if we set the radius as
\[ r_s \geq \underline{r}_{\delta,n} := \sqrt{\frac{2|\log\delta|c}{\gamma_s n}},\]
it is guaranteed that with probability at least $1-\delta$, we have $P \in \calP_n$.
\end{corollary}


To guarantee the existence of $r_s$ that satisfies the upper bound in Corollary~\ref{cor:radius_upper} and lower bound in Corollary~\ref{cor:radius_lower} at the same time, we give the following necessary condition on the minimum number of post-change training samples required.

\begin{lemma} \label{lem:good_prob}
Suppose the same conditions as in Corollary~\ref{cor:radius_upper} hold. Additionally, suppose
\begin{equation} \label{eq:min_n}
    n \geq \underline{n}_\delta := \frac{8|\log\delta|c}{\gamma_s (\wass_s(P,Q))^2},
\end{equation}
where $\underline{n}_\delta$ is the least number of samples to guarantee that $\overline{r}_{\delta,n}\geq \underline{r}_{\delta,n}$.
Now, if $r_s$ satisfies
\[ \underline{r}_{\delta,n} \leq r_s \leq \overline{r}_{\delta,n}, \]
then
\[ \mathbb P^P \brc{\{P \in \calP_n\} \cap \{Q \notin \calP_n\}} \geq 1 - 2 \delta. \]
\end{lemma}

\begin{proof}
We first note that when \eqref{eq:min_n} holds we have $\underline{r}_{\delta,n}\leq\overline{r}_{\delta,n}$. Then the result directly follows from the fact that for any two events $A,B$,
\[ \mathbb P (A \cap B) = \mathbb P(A) - \mathbb P(A \cap B^c) \geq \mathbb P(A) - \mathbb P(B^c). \]
Now, letting $A = \{P \in \calP_n\}$ and $B = \{Q \notin \calP_n\}$, and applying Corollaries~\ref{cor:radius_upper} and \ref{cor:radius_lower}, we get the desired result.
\end{proof}

In the following, we write the LFD for $n$ samples as $P_n^*$ and its pdf as $p_n^*$. Lemma~\ref{lem:radius} establishes an asymptotic upper bound on the worst-case detection delay of the DR-CuSum test,  for a radius $r_s$ that satisfies $\underline{r}_{\delta,n} \leq r_s \leq \overline{r}_{\delta,n}$.

\begin{lemma}\label{lem:radius}
Suppose that $\mathbb{E}^P \sbrc{\big(\log (p_n^\ast (X_1) / q(X_1))\big)^2} < \infty$. Fix $\delta\in(0,1)$ and $s \in [1,2]$. Suppose that the pre- and post-change distributions $Q$, $P$ are probability measures on $\mathbb R^d$, and they both satisfy the $T_s(c)$ inequality. 
Suppose that $n \geq (N_0 (r_s^{-(d+2)} \vee 1)) \vee \underline{n}_\delta$ where $N_0$ is the same as in Theorem~\ref{thm:wass_conc} and $\underline{n}_\delta$ is defined in \eqref{eq:min_n}. Then, if the chosen radius $r_s$ satisfies
\[ \underline{r}_{\delta,n} \leq r_s \leq \overline{r}_{\delta,n},\]
it is guaranteed that with probability at least $1 - 2\delta$,
the worst-case expected detection delay of the DR-CuSum test $\tau_{\mathrm{DR}}$ with threshold $b = \log \gamma$ can be upper bounded as
\begin{equation}\label{eq:wadd_upper}
   \sup_{P\in\calP_n} \WADD^{P} (\tau_{\mathrm{DR}})  
\leq \frac{\log \gamma}{\mathsf{KL}(P^*_n||Q)} \cdot (1+o(1)) \leq \frac{2c \log \gamma}{(\wass_s(P,Q) - 2r_s)^2 } \cdot (1+o(1)), 
\end{equation}
as $\gamma \to \infty$, where $P^*_n$ is the LFD in $\calP_n$, i.e., the solution to \eqref{eq:kl_min}.
\end{lemma}

\begin{proof}
Throughout the proof we use the result from Lemma~\ref{lem:good_prob}, i.e., under the given conditions, with probability at least $1 - 2 \delta$, we have $P \in \calP_n$ and $Q \notin \calP_n$.
To prove the first inequality in \eqref{eq:wadd_upper}, note that when $Q \notin \calP_n$ and for any $P \in \calP_n$, 
\[ \mathbb{E}^P \sbrc{\log \frac{p_n^\ast (X_1)}{q(X_1)} } = \mathsf{KL}(P||Q) - \mathsf{KL}(P||P^*_n) \stackrel{(i)}{\geq} \mathsf{KL}(P^*_n||Q) \stackrel{(ii)}{>} 0,\]
where $(i)$ follows from the WSB condition in Lemma~\ref{lemma:wsb}, and $(ii)$ follows from the fact that $Q \notin \calP_n$. Thus, by \cite[Prop.~8.21]{siegmund1985sequential}, $\mathbb{E}^P_1\sbrc{\tau_{\mathrm{DR}}} < \infty$.
By Wald's identity,
\[ \mathbb{E}^P \sbrc{\tau_{\mathrm{DR}}} = \frac{\mathbb{E}^P[S_{\tau_{\mathrm{DR}}}]}{\mathbb{E}^P \sbrc{\log \frac{p_n^\ast (X_1)}{q(X_1)}}} \leq \frac{\mathbb{E}^P[S_{\tau_{\mathrm{DR}}}]}{\mathsf{KL}(P^*_n||Q)} ,\quad \forall P \in \calP_n. \]
Since by assumption $\mathbb{E}^P \sbrc{\log \frac{p_n^\ast (X_1)}{q(X_1)}}^2 < \infty$, following classical renewal analysis (e.g., see \cite{siegmund1985sequential}), we have
\[ \mathbb{E}^P[S_{\tau_{\mathrm{DR}}}] = \log \gamma + \mathbb{E}^P[S_{\tau_{\mathrm{DR}}} - \log \gamma] = \log \gamma \cdot (1 + o(1)),\quad \text{as}~\gamma \to \infty.\]
Finally, we note that $\WADD^{P} (\tau_{\mathrm{DR}}) \leq \mathbb{E}^P_1 \sbrc{\tau_{\mathrm{DR}}},~\forall P \in \calP_n$, since $S_k \geq 0$ for all $k \geq 1$ almost surely and $S_k$ has a recursive update structure. This completes the proof of the first inequality.

For the second inequality in \eqref{eq:wadd_upper}, due to the triangle inequality for the Wasserstein metric, we have 
\begin{equation}\label{eq:was_ineq1}
\wass_s(\hatP_n,Q) \leq \wass_s(\hatP_n,P_n^*) +  \wass_s(P_n^*,Q) \stackrel{(iii)}{\leq}  r_s + \wass_s(P_n^*,Q),   
\end{equation}
where $(iii)$ is due to $P_n^*\in\calP_n$, and thus $ \wass_s(\hatP_n,P_n^*) \leq r_s$.
On the other hand, under the event $P \in \calP_n$, we have $\wass_s(\hatP_n,P) \leq r_s$, which implies 
\begin{equation}\label{eq:was_ineq2}
\wass_s(\hatP_n,Q) \geq \wass_s(P,Q)  - \wass_s(\hatP_n,P) \geq \wass_s(P,Q) - r_s.
\end{equation}
Combining \eqref{eq:was_ineq1} and \eqref{eq:was_ineq2} we have 
\begin{equation}\label{eq:was_ineq3}
\wass_s(P_n^*,Q) \geq \wass_s(\hatP_n,Q) - r_s \geq \wass_s(P,Q) - 2 r_s. 
\end{equation}
Further, since $Q$ satisfies the $T_s(c)$ inequality, we have
\[ \frac{(\wass_s(P_n^*,Q))^2}{2c} \leq \mathsf{KL}(P_n^*||Q). 
\]
Combining this with \eqref{eq:was_ineq3}, we obtain
\[ \mathsf{KL}(P_n^*||Q) \geq \frac{(\wass_s(P,Q) - 2r_s)^2}{2c}. \]
Substituting this into the first inequality in \eqref{eq:wadd_upper} completes the proof.
\end{proof}

\begin{example}
For the special case, where the pre- and post-change distributions are Gaussian, with $Q=N(\mu_0,1)$ and $P=N(\mu_1,1)$, we have $\mathsf{KL}(P || Q) = \frac12(\mu_1-\mu_0)^2 = \frac12 \wass_2^2(Q,P)$, and the $T_2(1)$ inequality holds equality. This means that the delay of the DR-CuSum procedure is, with probability at least $1-2\delta$, bounded from above as
\begin{align*}
\frac{\log \gamma}{\mathsf{KL}(P_n^\ast || Q)} (1+o(1)) & \leq \frac{\log \gamma}{\frac{(\wass_2(P,Q) - 2r_2)^2}{2} } (1+o(1)) \\
& = \frac{\log \gamma}{\mathsf{KL}(P || Q) -2r_2 \sqrt{2\mathsf{KL}(P || Q)} + 2r_2^2 } (1+o(1)).
\end{align*}
From Corollary\,\ref{cor:radius_lower}, we may choose the radius as its lower bound with $r_2 = \sqrt{2|\log\delta|c'/(\gamma_2 n)} = O(n^{-1/2})$, which means that for $n$ sufficiently large, we have that the delay of DR-CuSum test will match the optimal delay, $[(\log\gamma)/\mathsf{KL}(P || Q)]\cdot(1+o(1))$, asymptotically.
\end{example}

\section{Minimax Robust QCD under Wasserstein Uncertainty Sets: Multiple Post-Change Uncertainty Sets} \label{sec:multiple_pcs}

We extend the results of Section~\ref{sec:single_pcs} to the more general case with multiple post-change scenarios as follows. 
Suppose there are $M\geq 1$ potential post-change scenarios in total, and we have a set of training samples $\{\omega_1^{(m)},\ldots,\omega_{n_m}^{(m)}\}$ that are independently sampled from the $m$-th post-change scenario, with $\hatP_{n_m}^{(m)}$ being the corresponding empirical distribution.
The uncertainty set $\calP_{n_m}^{(m)}$ which corresponds to the $m$-th post-change scenario, similar to \eqref{eq:uncertainty}, is now defined as
\begin{equation}\label{eq:uncertainty_multi}
\begin{aligned}
\calP_{n_m}^{(m)} &:= \{P\in \scrP_s: \wass_s(P,\hatP_{n_m}^{(m)}) \leq r_{s,m} \}, \quad m =1,\ldots,M,
\end{aligned}
\end{equation}
where $r_{s,m} \geq 0$ is the radius parameter controlling the size of the $m$-th uncertainty set. With a slight abuse of notation, we define $\calP := \cup_{m=1}^M \calP_{n_m}^{(m)}$ as the union of all the uncertainty sets in the remainder of this section.

\subsection{Asymptotically Optimal Stopping Time}

Based on Theorem\,\ref{thm:lfd}, we can find $M$ LFDs,  denoted, with a slight abuse of notation, as $P_{(1)}^*,\ldots,P_{(M)}^*$, one for each Wasserstein uncertainty set. The LFD $P_{(m)}^\ast$ for the $m$-th uncertainty set is an exponential tilting of $Q$ and has pdf $p_{(m)}^*(x) = q(x)\exp\{-C_{\lambda_m^*,u_m^*}^{(m)}(x)-\eta^{(m)}(\lambda_m^*,u_m^*) \}$, where $\lambda_m^*,u_m^*$ are the solution to
\[
\sup_{\substack{\lambda\geq 0, u \in \mathbb R^{n_m}}} \bigg\{ -\lambda r_{s,m}^s +  \frac{1}{n_m}\sum_{j=1}^{n_m} u_j  -\log \eta^{(m)}(\lambda,u) \bigg\},   
\]
where $
C^{(m)}_{\lambda,u}(x):= \min_{1\leq j\leq n_m} \{\lambda c^s(x,\omega_j^{(m)}) - u_j\}$ and 
\[
\eta^{(m)}(\lambda,u) :=\int  q(x) \exp\{-C^{(m)}_{\lambda,u}(x)\}d\mu(x).
\]
The log-likelihood ratio under scenario $m$ equals
\[
\log \frac{p_{(m)}^\ast(x)}{q(x)}  = -C^{(m)}_{\lambda_m^*,u_m^*}(x)-\log \eta^{(m)}(\lambda_m^*,u_m^*).
\]
Given online samples $\{X_k,\ k\in \bN \}$, the detection statistic for the $m$-th uncertainty set can be computed recursively as
\begin{equation}\label{eq:CuSum}
S_k^{(m)} = (S_{k-1}^{(m)})^{+} + \log \frac{p_{(m)}^*(X_k)}{q(X_k)},\quad S_0^{(m)} = 0,\quad \forall m=1\dots,M.
\end{equation}
The \textit{Distributionally Robust (DR) CuSum} stopping time under multiple post-change scenarios is 
then defined as
\begin{equation}\label{eq:CuSum_stop_multi}
\tau_{\mathrm{DR}}(b) := \inf\left\{k \in \bN:  \max_{m=1,\ldots,M} S_k^{(m)} \geq b \right\},
\end{equation}
where $b$ is chosen to meet the false alarm constraint.
In Lemma~\ref{lemma:fa} and Theorem~\ref{thm:opt}, we investigate the asymptotic optimality properties of this DR-CuSum test. The proofs of these results are provided in the Appendix.

\begin{lemma} \label{lemma:fa}
The mean time to false alarm of the test in  \eqref{eq:CuSum_stop_multi} satisfies 
\[
\mathbb E_\infty[\tau_{\mathrm{DR}}(b)] \geq \frac{e^b}{M}.
\]
\end{lemma}

\begin{theorem}[Asymptotic Minimax Robustness]\label{thm:opt}
Write 
\[
I^* := \min_{m=1,2,\ldots,M} \mathsf{KL}(P_{(m)}^*||Q).
\]
Then, the test in \eqref{eq:CuSum_stop_multi} with threshold $b_\gamma=\log (M \gamma)$ solves the problem in \eqref{eq:Lorden} 
asymptotically as $\gamma \rightarrow \infty$, with the asymptotic worst-case delay being
\begin{align*}
   \sup_{P\in\calP} \WADD^{P} (\tau_{\mathrm{DR}}(b_\gamma))  &=  \inf_{\tau' \in C(\gamma)} \sup_{P\in\calP} \WADD^{P} (\tau') \cdot (1+o(1)) \\
   &
     = \frac{\log \gamma}{I^*} \cdot (1+o(1)).
\end{align*}
\end{theorem}

\section{Numerical Results}\label{sec:num_res}

\subsection{Synthetic Data Examples}

We validate the performance of the DR-CuSum test \eqref{eq:CuSum_stop_multi} through simulations for an example in which the observations come from Gaussian distributions. We use the cost function $c(x,x') = \norm{x-x'}_2$ and order $s=2$ for calculating the  Wasserstein distance. The true pre- and post-change distributions are $\mathcal N(0,1)$ and $\mathcal N(0.5,1)$, respectively. Therefore, the Wasserstein distance between the true pre- and post-change distribution is $\wass_2(\mathcal{N}(0,1),\mathcal{N}(0.5,1)) = (0.5-0)^2 = 0.25$.


\paragraph{Comparison with CuSum Type Tests and Effect of Radius:} 

We simulate the case of a single post-change scenario ($M=1$). 
We first compare the performances for the following three tests that all have a recursive structure that facilitates implementation:
\begin{enumerate}
    \item The CuSum test for the case where both the pre- and post-change distributions are known. This is the optimal procedure and provides us with an information-theoretic lower bound (benchmark) for the WADD.
    \item The CuSum test that has knowledge of the Gaussian post-change model, and uses the training data to produce a MLE of the post-change mean {\it and} variance. 
    \item The proposed DR-CuSum test defined in \eqref{eq:CuSum_stop_multi}, with different choices of radius.
\end{enumerate}
Since all four tests use statistics with a recursive CuSum structure and the observations are independent, the worst-case value of the change-point for computing the WADD is $\nu=1$ for all tests. This allows us to estimate the worst-case delays of the tests by simulating the post-change distribution from time 1. Furthermore, the recursive nature of these four tests implies that they have similar computational complexities during the detection phase.

\begin{figure}[htbp!]
    \centering
    \includegraphics[width=0.8\linewidth]{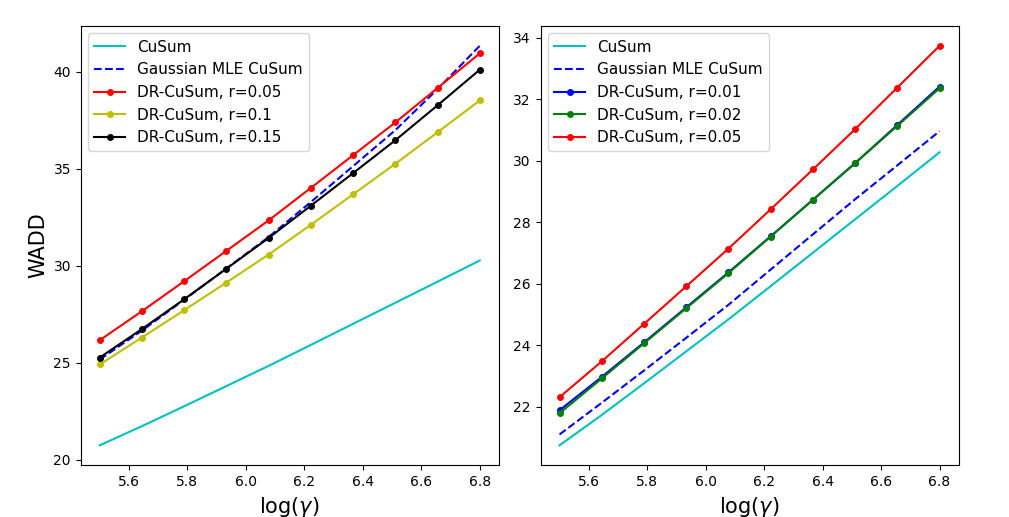}
    \caption{Performances of tests with different statistics, averaged over 30 sets of different training samples. The left plot corresponds to the case with $n=25$ post-change training samples, and the right plot corresponds to that with $n=150$. The tests shown are regular CuSum (cyan), CuSum with Gaussian MLE (blue dashes), and DR-CuSum with various radii.}
    \label{fig:gauss_perf} 
\end{figure}
\begin{figure}[htbp!]
    \centering
    \includegraphics[width=0.6\linewidth]{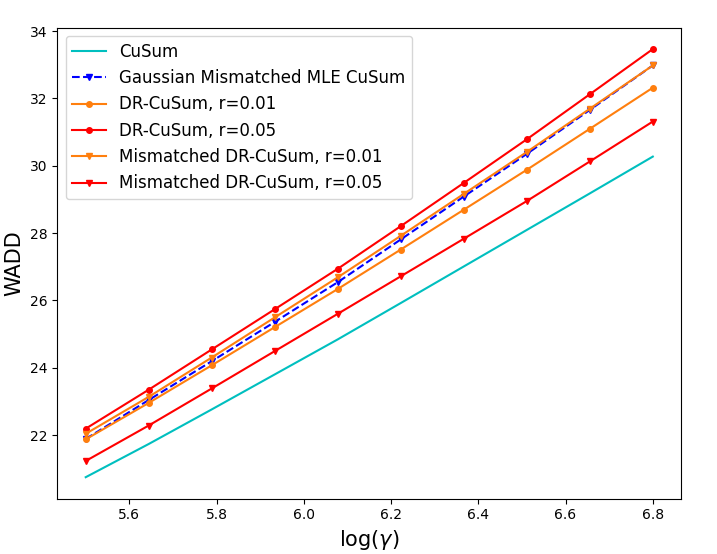}
    \caption{Performances of tests with different mismatched statistics, averaged over 40 sets of $n=150$ post-change training samples. The samples are generated from $\mathcal N(0.75,1)$, which is different from the true post-change distribution $\mathcal N(0.5,1)$. The tests shown are regular CuSum (cyan), mismatched CuSum with Gaussian MLE (blue dashes), and both matched (in circle markers) and mismatched (in triangle markers) DR-CuSum tests with two radii selected.}
    \label{fig:gauss_mismatch_perf}
\end{figure}

In Fig.~\ref{fig:gauss_perf}, we numerically study the effect of radius under two conditions for the number of post-change training samples \textit{a priori}: when the sample size is small ($n=25$) and when the sample size is large ($n=150$).
We see in Fig.~\ref{fig:gauss_perf} that the choice of radius in DR-CuSum is important for its performance. When the number of training samples is small, the DR-CuSum test outperforms the Gaussian MLE CuSum test with several choices of radii. We emphasize that, unlike the latter test, the DR-CuSum test does {\it not} assume any knowledge of the parametric model for the post-change distribution. This highlights the effectiveness of the DR-CuSum test in dealing with distributional uncertainty, especially in data-driven and non-parametric settings.


In Fig.~\ref{fig:gauss_mismatch_perf}, we numerically study the effect of radius when the number of mismatched post-change samples is plenty ($n=150$). The empirical samples are drawn from a mismatched Gaussian distribution: $\mathcal N(0.75,1)$, while the true post-change distribution for test sequences is still $\mathcal N(0.5,1)$.
We see in Fig.~\ref{fig:gauss_mismatch_perf} that the model mismatch causes a non-trivial effect on the optimal radius selection, where the DR-CuSum test with a larger radius is more robust under distributional mismatch. Also, with a proper choice of radius, we see that the mismatched DR-CuSum test outperforms the mismatched Gaussian MLE test, which, we again emphasize, knows the parametric model for the post-change distribution. This highlights the effectiveness of DR-CuSum test in dealing with training data mismatch, which is common in data-driven applications.
More results for the comparison of different radii can be found in Appendix~\ref{sec:numerical_appendix}.

\paragraph{Effect of Sample Size:}

\begin{figure}[htbp!]
    \centering
    \includegraphics[width=0.6\linewidth]{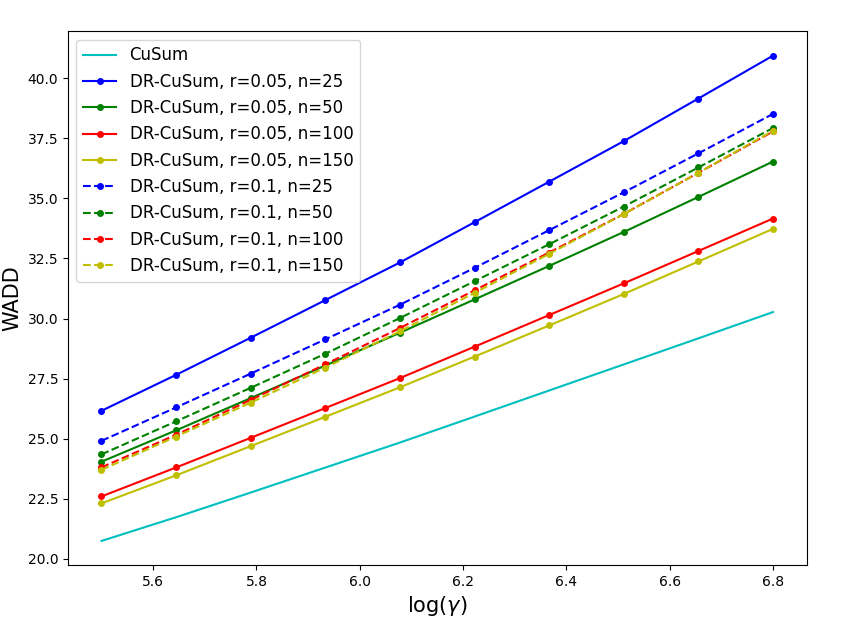}
    \caption{Performances of DR-CuSum tests with different numbers of post-change training samples, averaged over 30 different training sets. The tests shown are regular CuSum (cyan) and the DR-CuSum with radii $r=0.05$ (solid lines) and $r=0.1$ (dashed lines) at varying post-change training sample sizes.}
    \label{fig:gauss_radius} 
\end{figure}

We see in Fig.~\ref{fig:gauss_radius} that an increase in the number of training samples is always beneficial to the performance of DR-CuSum. Also, this benefit is greater when the choice of radius is small. Therefore, the optimal radius needs to be smaller when the number of post-change training samples becomes larger.

\paragraph{Effect of $M$:}

\begin{figure}[ht!]
    \centering
    \includegraphics[width=0.6\linewidth]{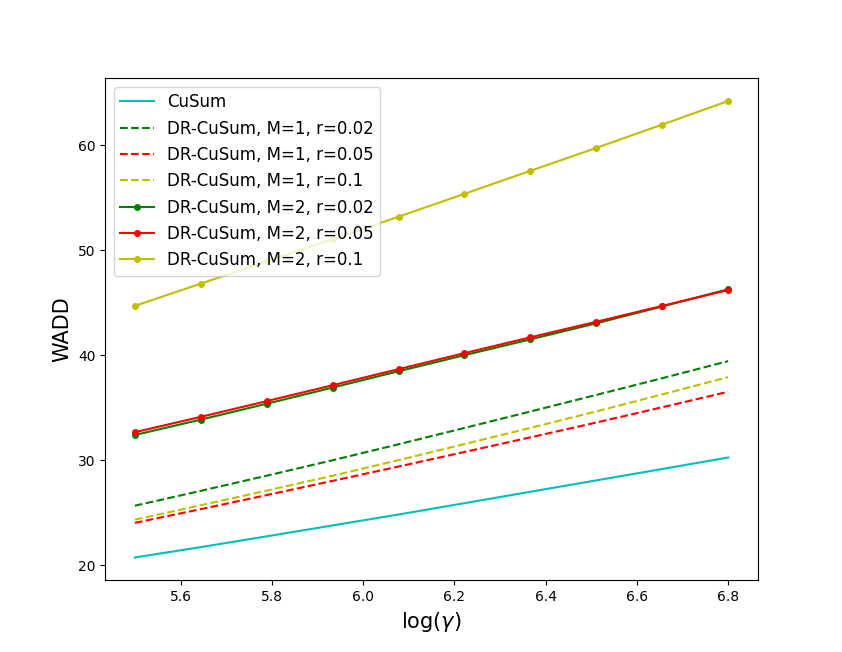}
    \caption{Performances of DR-CuSum tests with different numbers of scenarios, averaged over 30 different training sets. The tests shown are regular CuSum (cyan), the DR-CuSum test that knows the true scenario (denoted as $M=1$), and the multi-scenario DR-CuSum test ($M=2$). $n=50$. The DR-CuSum tests are studied at different radii.}
    \label{fig:gauss_multi} 
\end{figure}

We also study the performance of DR-CuSum under multiple post-change scenarios. The total number of scenarios is $M=2$, and each scenario is given $n_1 = n_2 =50$ training samples. We see in Fig.~\ref{fig:gauss_multi} that the uncertainty in the true scenario results in some performance loss for the DR-CuSum test. Also, the optimal choice of radius depends on the value of $M$.

\paragraph{Comparison with NGLR-CuSum test:}


When the post-change training sample size is small, we compare the performance of the DR-CuSum test with that of the NGLR-CuSum test \citep{liang2023quickest}. Since the NGLR-CuSum test in \cite{liang2023quickest} does not use any post-change training samples, we define a modified version that uses the post-change training samples.
The change-point $\nu = 1$, and note that, different from the DR-CuSum test, this choice does not guarantee a worst-case delay for the NGLR-CuSum test.


\begin{figure}[htbp!]
    \centering
    \begin{tabular}{cc}
      \includegraphics[width=0.6\textwidth]{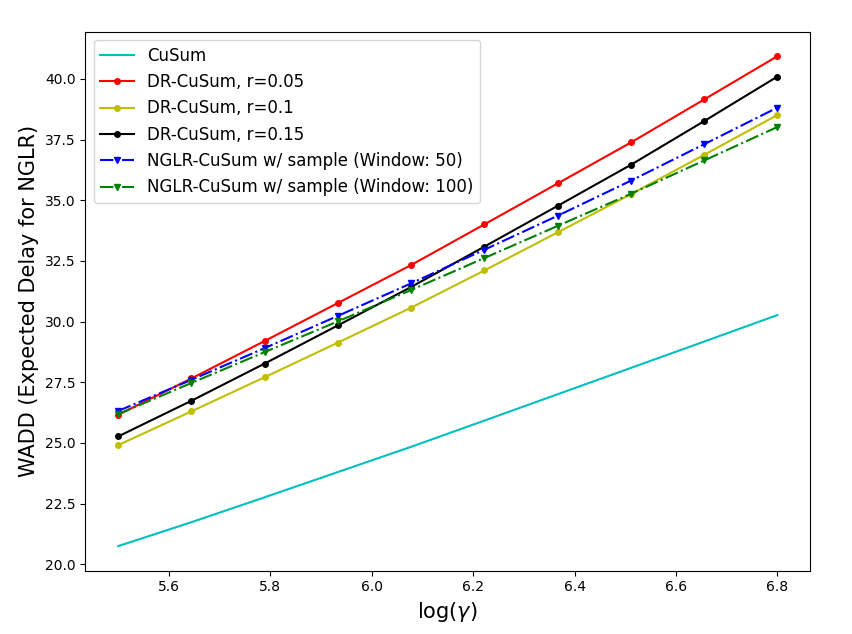}  
\end{tabular}
    \caption{Comparison of DR-CuSum tests (solid lines) with the modified NGLR-CuSum tests (dashed lines) defined in \eqref{eq:def_nglr_with_sample}. The number of post-change training samples $n = 25$. The KDE with a Gaussian kernel is used in the NGLR-CuSum test, with the bandwidth parameter $h = 50^{-0.2}$. All tests are first evaluated on the same set of post-change training samples, and then the average performance over 30 different sets of training samples is reported. }
    \label{fig:gauss_vs_loo}
\end{figure}

The NGLR-CuSum test in \cite{liang2023quickest} is defined as
\begin{equation} \label{eq:def_nglr}
    \tau_{\rm{NGLR}}(b) := \inf \left\{k \geq 1:\max_{(k-W)^+ < \ell \leq k} \sum_{j=\ell}^k \log \frac{\widehat{p}^{k,\ell}_{-j}(X_j)}{q(X_j)} \geq b \right\},
\end{equation}
where $W$ is the window size, and if we assume using a kernel density estimator (KDE) with some kernel function $K(\cdot)$ and bandwidth $h$ \citep{wasserman2006all-nonpara-stat}, the leave-one-out density estimate is
\[ \widehat{p}^{k,\ell}_{-j}(X_j) := \frac{1}{(k-\ell) h} \sum_{\substack{i=\ell \\ i \neq j}}^{k} K\left(\frac{X_i-X_j}{h}\right),\quad \forall j \in [\ell,k]. \]
To utilize the post-change training samples, we similarly define the modified NGLR-CuSum test as:
\begin{equation} \label{eq:def_nglr_with_sample}
    \tau_{\rm{NGLRws}}(b) := \inf \left\{k \geq 1:\max_{(k-W)^+ < \ell \leq k} \brc{\sum_{j=\ell}^k \log \frac{\widehat{p}^{k,\ell,\omega}_{-j}(X_j)}{q(X_j)} + \sum_{i=1}^n \log \frac{\widehat{p}^{k,\ell,\omega}_{-j}(\omega_i)}{q(\omega_i)}}\geq b \right\},
\end{equation}
where
\[ \widehat{p}^{k,\ell,\omega}_{-j}(X_j) := \frac{1}{(k-\ell+n) h} \brc{ \sum_{\substack{i=\ell \\ i \neq j}}^{k} K\left(\frac{X_i-X_j}{h}\right) + \sum_{i=1}^{n} K\left(\frac{\omega_i-X_j}{h}\right) },~\forall j \in [\ell,k]. \]
In Fig.~\ref{fig:gauss_vs_loo}, we see that given the training samples, the DR-CuSum test (with the optimal radius) performs slightly worse than the modified NGLR-CuSum test. However, due to the recursive CuSum update structure, the DR-CuSum test is computationally less expensive than the latter at inference time. 


\begin{figure}[t!]
    \centering
    \begin{tabular}{cc}
      \includegraphics[width=0.6\textwidth]{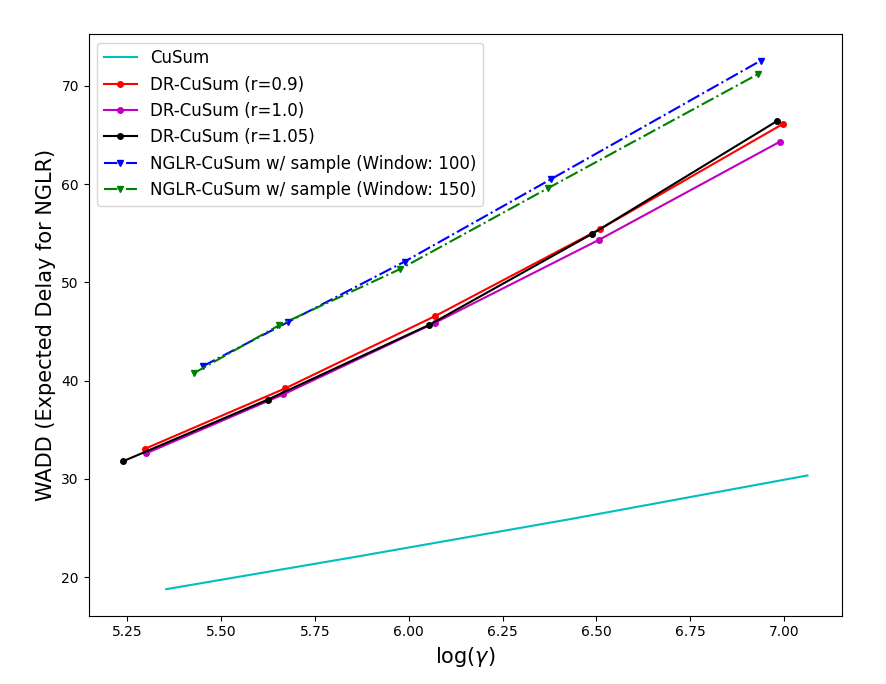}  
\end{tabular}
    \caption{Comparison of DR-CuSum tests (solid lines) with the modified NGLR-CuSum tests (dashed lines) defined in \eqref{eq:def_nglr_with_sample} with $d = 3$ and $a = 0.3$. The number of post-change training samples $n = 25$. The leave-one-out KDE with a Gaussian product kernel (defined in \eqref{eq:def_product_kernel}) is used in the NGLR-CuSum test, with the bandwidth parameter $h_m = 30^{-1/7},\forall m = 1,\dots,3$. All tests are first evaluated on the same set of post-change training samples, and then the average performance over 30 different sets of training samples is reported. }
    \label{fig:gauss_vs_loo_3d}
\end{figure}

We also compare the performance of the DR-CuSum tests with that of the NGLR-CuSum test with $d$ dimensional observations. The pre-change distribution is $\mathcal N(\bm{0},\bm{\mathrm{I}}_d)$, where $\bm{\mathrm{I}}_d$ is the identity matrix. The post-change distribution is $\mathcal N(\bm{a},\bm{\mathrm{I}}_d)$. Here $\bm{a} \in \mathbb{R}^d$ denotes a vector with all elements being $a \in \mathbb{R}$. The change-point $\nu = 1$.
We use the following product kernel in the leave-one-out density estimate. For any $j \in [\ell,k]$,
\begin{multline} \label{eq:def_product_kernel}
    \widehat{p}^{k,\ell,\omega}_{-j}(\bm{X}_j) := \frac{1}{(k-\ell+n) \prod_{m=1}^d h_{m}} \times \\
    \brc{ \sum_{\substack{i=\ell \\ i \neq j}}^{k} \prod_{m=1}^d K\left(\frac{\bm{X}^{(m)}_i-\bm{X}^{(m)}_j}{h_m}\right) + \sum_{i=1}^{n} \prod_{m=1}^d K\left(\frac{\bm{\omega}^{(m)}_i-\bm{X}^{(m)}_j}{h_m}\right) }.
\end{multline}
where $\bm{x}^{(m)}$ denotes the $m$-th element of vector $x$ and $h_m$ denotes the kernel bandwidth for the $m$-th element.

In Fig.~\ref{fig:gauss_vs_loo_3d}, we see that with 3-d observations, the DR-CuSum test (with the optimal radius) performs better than the modified NGLR-CuSum test. This is because kernel density estimation becomes less accurate in higher dimensions.
Also, it is observed that the DR-CuSum test is computationally much less expensive than the modified NGLR-CuSum test. The kernel density estimation is very computationally demanding in higher dimensions.
In comparison, while the DR-CuSum test also suffers from a more expensive offline computation, its online computational requirements only go up modestly due to the increase in dimension. 
Indeed, the DR-CuSum requires only $O(n d)$ operations to compute $C_{\lambda^*,u^*}(\bm{X}_k)$ for each new sample $\bm{X}_k$.


\subsection{Real Data Example}


We apply the proposed DR-CuSum test to a real data example of human activity detection. We use the WISDM's Actitracker activity prediction dataset \citep{lockhart2011design}. This dataset contains the activity data collected from 225 users under six different activities, including ``Walking'', ``Jogging'', ``Stairs'', ``Sitting'', ``Standing'', and ``LyingDown''. The attribute at each time is a three-dimensional vector containing the acceleration in $x$-, $y$-, and $z$-axes. We select ``Walking'' as the nominal pre-change state and our goal is to detect a change to the ``Jogging'' state (post-change) as quickly as possible. 

We mainly compare the proposed DR-CuSum with the NGLR-CuSum test. For the NGLR-CuSum, we first fit a Gaussian distribution as the pre-change distribution using available pre-change samples from several users whose pre-change sample size is large. Then we apply the test in \eqref{eq:def_nglr_with_sample} with product kernel \eqref{eq:def_product_kernel}.
For the DR-CuSum test, following  Remark~\ref{rmk:emp_pre}, we directly solve for the LFD $P^*$ using the pre-change samples without estimating the pre-change density. We first visualize the trajectory of the DR-CuSum detection statistics for a particular user. Then we provide the comparison of average detection delay (over multiple users) at the end of this subsection.

\begin{figure}[htbp!]
    \centering
    \begin{tabular}{cc}
      \includegraphics[width=0.45\textwidth]{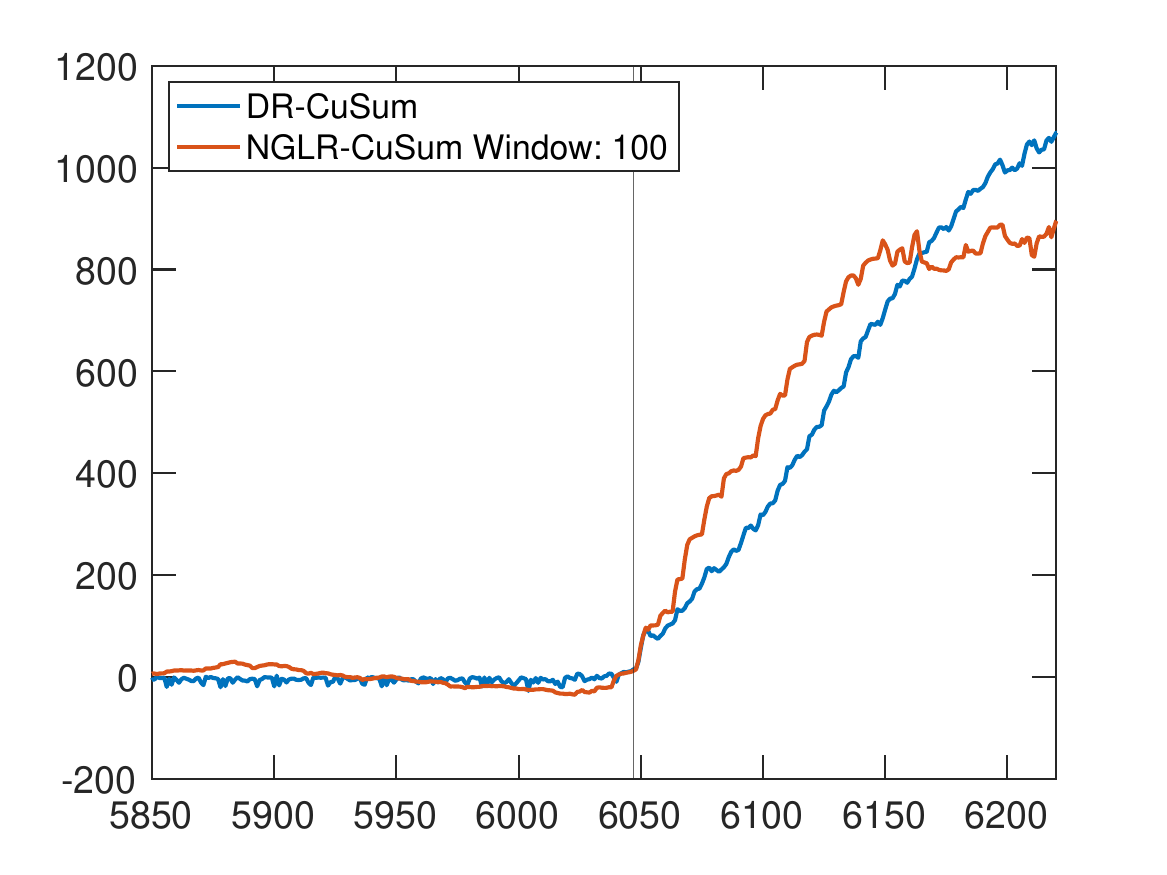}   &  \includegraphics[width=0.45\textwidth]{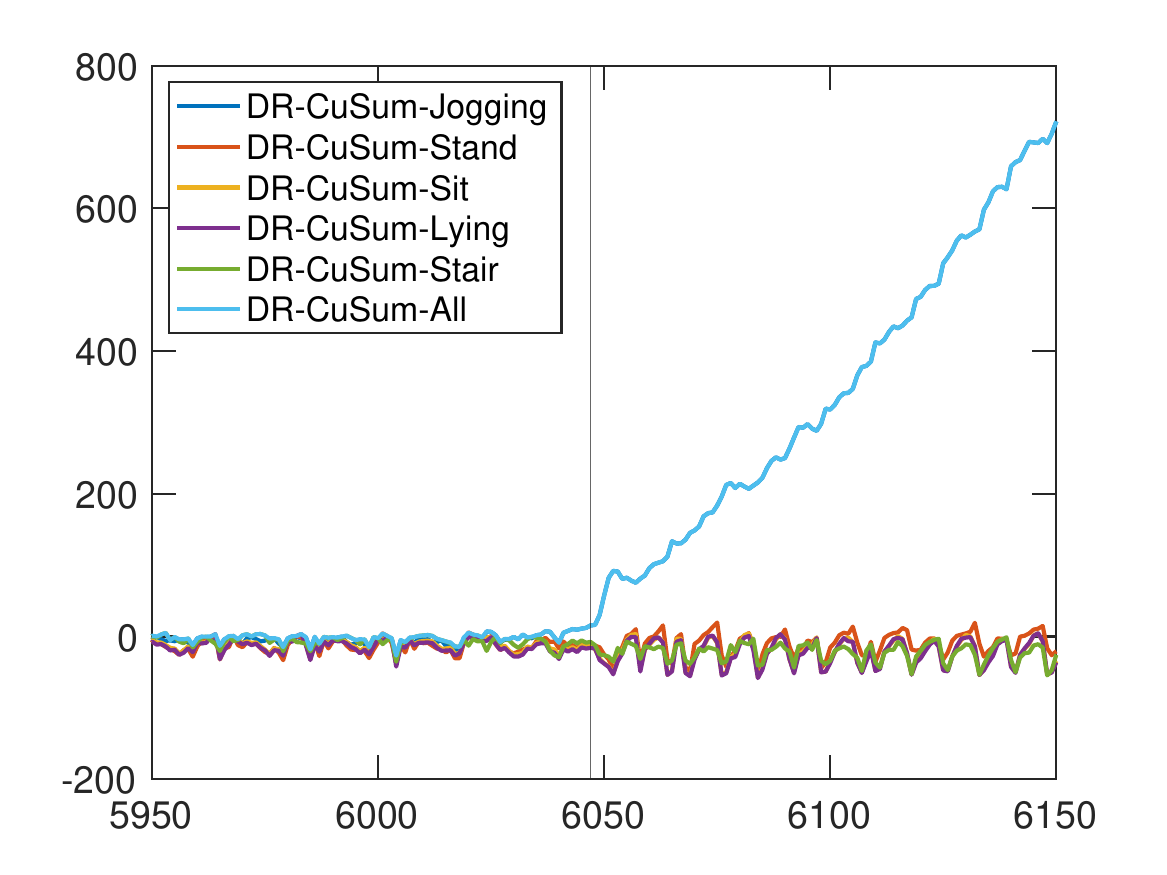}  
\end{tabular}
    \caption{Activity monitoring using DR-CuSum test, where the goal is to make an alarm when the state changes from Walking to Jogging. The pre-change distribution for NGLR-CuSum test is assumed to be Gaussian and estimated using historical data. For each of the Jogging, Stairs, Sitting, Standing, and LyingDown states, empirical samples are taken from one user's data to construct all three tests. Monitoring is performed on another user's activity. Number of empirical samples from each scenario $n=5$. Wasserstein order $s=1$. Radius $r=8$. Number of scenarios $M=5$ (right plot only). The vertical line represents the true change-point.}
    \label{fig:activity_CuSum_5}
\end{figure}

In the first scenario, we consider $M=1$ with ``Jogging'' being the true post-change activity. We select $n=5$ samples from the Jogging state of a specific user to construct the empirical distribution $\widehat P$. This represents the case where number of training samples available is very small. The DR-CuSum test, along with the baseline NGLR-CuSum test, is applied on \textit{another} user's data to monitor his/her activity change from Walking to Jogging. We emphasize that here the choice of empirical samples can be flexible, as we did not require post-change samples to come from the target user whose activity is being detected, but we instead use empirical samples obtained from a {\it different} user to construct the uncertainty set. This is reasonable in the activity detection example as the entire dataset has 225 users in total. Moreover, this also demonstrates the importance of constructing an uncertainty set for detection since the distribution of Jogging among different users can have some minor differences. Fig.~\ref{fig:activity_CuSum_5} (Left) shows an example where there is a clearer dichotomy for the DR-CuSum statistic before and after the change, and the DR-CuSum statistics is more stable under the pre-change regime. 

In the second scenario, we consider $M=5$ with the post-change activity being one of the Jogging, Stairs, Sitting, Standing, and LyingDown state. We construct the $M$ uncertainty sets using $n=5$ empirical samples from each state and solve for LFDs $P_1^*,\ldots,P_M^*$. The resulting DR-CuSum detection statistics in Fig.~\ref{fig:activity_CuSum_5} (Right) 
show that the maximum statistic in \eqref{eq:CuSum_stop_multi} is helpful not only for change detection, but also for change isolation.



We also compare the average detection delay of DR-CuSum test with that of the NGLR-CuSum test. We focus on the detection of activity change from Walking to Jogging and select 86 user sequences that contain such activity change. For each of these 86 sequences, we use $n$ empirical data randomly selected from the post-change state to construct the uncertainty set for the DR-CuSum test, and for density estimation in \eqref{eq:def_nglr_with_sample}. We repeat such procedures ten times for each user to account for the randomness in selecting the post-change empirical samples. The average detection delay and standard deviation are reported in Table~\ref{tab:ADD}. We can see that the DR-CuSum test tends to have a smaller detection delay and a smaller variability compared with NGLR-CuSum. 

\begin{table}[htbp!]
    \centering
    \caption{Average detection delay of DR-CuSum and NGLR-CuSum test on 86 user sequences for detecting the change from Walking to Jogging. The bandwidth is selected using the rule of thumb as $h_i=W^{-1/(d+4)}\hat\sigma_i$, where $W=100$ is the window size, $d=3$ is the data dimension, and $\hat\sigma_i$ is the estimated standard deviation from pre-change data. The window for NGLR-CuSum is set as $W=100$. The threshold is chosen as the upper 1\% quantile of the detection statistics for the pre-change samples. The experiments are repeated ten times and the average detection delay is reported, with standard deviations in parentheses.}
    \begin{tabular}{lcc}
    \hline 
         & DR-CuSum & NGLR-CuSum \\
         \hline
    $n=10$, $r=2$  & 25.61 (4.08) & 81.26 (0.51)   \\
    $n=20$, $r=1.5$ & 16.88 (2.74)   &   74.77 (13.62)  \\
    \hline
    \end{tabular}
    \label{tab:ADD}
\end{table}

\section{Conclusion}\label{sec:conclusion}
We developed an asymptotically minimax robust procedure for QCD, which we refer to as the distributionally robust (DR) CuSum test, in the setting where the pre-change distribution is known, and where the post-change distribution belongs to a union of data-driven Wasserstein uncertainty sets. We showed that the density corresponding to the post-change least favorable distribution within the Wasserstein uncertainty set is an exponential tilting of the known pre-change density. We characterized the minimum sufficient radius for the uncertainty set through concentration inequalities for empirical measures in Wasserstein distance. For an example with Gaussian observations, we compared through Monte-Carlo simulations, the performance of the DR-CuSum test with that of the CuSum test that has knowledge of the Gaussian post-change model and uses the maximum likelihood estimate (MLE) of the post-change mean and variance from the labeled data. We showed that with an appropriately chosen radius, the DR-CuSum test, which makes no distributional assumptions about the post-change, outperforms the Gaussian MLE CuSum test. 
We also numerically compared the performance of the DR-CuSum test with that of the modified NGLR-CuSum test \citep{liang2023quickest}, and found that with 3-d observations the DR-CuSum test outperforms NGLR-CuSum in terms of both computational complexity and delay performance. 
We also demonstrated the performance of the DR-CuSum test on a real human activity detection dataset.
Our theoretical findings can be extended to the non-stationary setting where the post-change observations are independent but not necessarily identically distributed using \cite[Theorem II.3]{liang2021non}; we leave this extension for future research. 

\section*{Funding Acknowledgement}
The work of Liyan Xie was partially supported by UDF01002142 and 2023SC0019 through the Chinese University of Hong Kong, Shenzhen. The work of Yuchen Liang and Venugopal V. Veeravalli was supported by the U.S. National Science Foundation under grant  ECCS-2033900, and by the U.S. Army Research Laboratory under Cooperative Agreement W911NF-17-2-0196, through the University of Illinois at Urbana-Champaign.

\bibliographystyle{tfcad}
\bibliography{ref}

\begin{thebibliography}{37}
\newcommand{\enquote}[1]{``#1''}
\providecommand{\natexlab}[1]{#1}
\providecommand{\url}[1]{\normalfont{#1}}
\providecommand{\urlprefix}{}

\bibitem[Basseville and Nikiforov(1993)]{basseville1993detection}
Basseville, Mich{\`e}le, and Igor~V Nikiforov. 1993. \emph{{Detection of Abrupt
  Changes: Theory and Application}}. Vol. 104. Prentice Hall Englewood Cliffs.

\bibitem[Bolley, Guillin, and Villani(2007)]{bolley2007quantitative}
Bolley, Fran{\c{c}}ois, Arnaud Guillin, and C{\'e}dric Villani. 2007.
  ``Quantitative concentration inequalities for empirical measures on
  non-compact spaces.'' \emph{Probability Theory and Related Fields} 137:
  541--593.

\bibitem[Boyd and Vandenberghe(2004)]{boyd2004convex}
Boyd, Stephen, and Lieven Vandenberghe. 2004. \emph{Convex Optimization}.
  Cambridge university press.

\bibitem[Gao et~al.(2018)]{gao2018robust}
Gao, Rui, Liyan Xie, Yao Xie, and Huan Xu. 2018. ``Robust hypothesis testing
  using {W}asserstein uncertainty sets.'' In \emph{Proceedings of the Advances
  in Neural Information Processing Systems (NeurIPS)}, 7902--7912.

\bibitem[Huber(1965)]{huber1965}
Huber, Peter~J. 1965. ``A robust version of the probability ratio test.''
  \emph{Annals of Mathematical Statistics} 36 (6): 1753--1758.

\bibitem[Kawahara and Sugiyama(2012)]{sugiyama2012direct}
Kawahara, Yoshinobu, and Masashi Sugiyama. 2012. ``Sequential change-point
  detection based on direct density-ratio estimation.'' \emph{Statistical
  Analysis and Data Mining: The ASA Data Science Journal} 5 (2): 114--127.
  \urlprefix\url{https://onlinelibrary.wiley.com/doi/abs/10.1002/sam.10124}.

\bibitem[Kuhn et~al.(2019)]{kuhn2019wasserstein}
Kuhn, Daniel, Peyman~Mohajerin Esfahani, Viet~Anh Nguyen, and Soroosh
  Shafieezadeh-Abadeh. 2019. ``Wasserstein distributionally robust
  optimization: Theory and applications in machine learning.'' In
  \emph{Operations research \& management science in the age of analytics},
  130--166. Informs.

\bibitem[Lai(1998)]{lai-ieeetit-1998}
Lai, Tze~Leung. 1998. ``Information bounds and quick detection of parameter
  changes in stochastic systems.'' \emph{IEEE Transactions on Information
  Theory} 44 (7): 2917--2929.

\bibitem[Lai and Xing(2010)]{lai2010sequential}
Lai, Tze~Leung, and Haipeng Xing. 2010. ``Sequential change-point detection
  when the pre- and post-change parameters are unknown.'' \emph{Sequential
  Analysis} 29 (2): 162--175.

\bibitem[Levy(2008)]{levy2008principles}
Levy, Bernard~C. 2008. \emph{Principles of Signal Detection and Parameter
  Estimation}. Springer Science \& Business Media.

\bibitem[Li et~al.(2015)]{xie2015mstat}
Li, Shuang, Yao Xie, Hanjun Dai, and Le~Song. 2015. ``M-statistic for kernel
  change-point detection.'' In \emph{Advances in Neural Information Processing
  Systems},  edited by C.~Cortes, N.~Lawrence, D.~Lee, M.~Sugiyama, and
  R.~Garnett, Vol.~28, 3366--3374. Curran Associates, Inc.

\bibitem[Liang and Veeravalli(2022)]{liang2021non}
Liang, Yuchen, and Venugopal~V Veeravalli. 2022. ``Non-parametric quickest
  mean-change detection.'' \emph{IEEE Transactions on Information Theory} 68
  (12): 8040--8052.

\bibitem[Liang and Veeravalli(2023)]{liang2023quickest}
Liang, Yuchen, and Venugopal~V Veeravalli. 2023. ``Quickest change detection
  with leave-one-out density estimation.'' In \emph{Proceedings of the IEEE
  International Conference on Acoustics, Speech and Signal Processing
  (ICASSP)}, 1--5. IEEE.

\bibitem[Lockhart et~al.(2011)]{lockhart2011design}
Lockhart, Jeffrey~W, Gary~M Weiss, Jack~C Xue, Shaun~T Gallagher, Andrew~B
  Grosner, and Tony~T Pulickal. 2011. ``Design considerations for the {WISDM}
  smart phone-based sensor mining architecture.'' In \emph{Proceedings of the
  Fifth International Workshop on Knowledge Discovery from Sensor Data},
  25--33. ACM.

\bibitem[Lorden(1971)]{lorden1971}
Lorden, Gary. 1971. ``Procedures for reacting to a change in distribution.''
  \emph{Annals of Mathematical Statistics} 42 (6): 1897--1908.

\bibitem[Luenberger(1969)]{luenberger1969optimization}
Luenberger, David~G. 1969. \emph{Optimization by Vector Space Methods}. Series
  in decision and control. 605 Third Ave, New York, NY, US: John Wiley \& Sons,
  Inc.

\bibitem[Magesh et~al.(2023)]{magesh2023robust}
Magesh, Akshayaa, Zhongchang Sun, Venugopal~V Veeravalli, and Shaofeng Zou.
  2023. ``Robust hypothesis testing with moment constrained uncertainty sets.''
  In \emph{ICASSP 2023-2023 IEEE International Conference on Acoustics, Speech
  and Signal Processing (ICASSP)}, 1--5. IEEE.

\bibitem[Mohajerin~Esfahani and Kuhn(2018)]{mohajerin2018data}
Mohajerin~Esfahani, Peyman, and Daniel Kuhn. 2018. ``Data-driven
  distributionally robust optimization using the Wasserstein metric:
  Performance guarantees and tractable reformulations.'' \emph{Mathematical
  Programming} 171 (1-2): 115--166.

\bibitem[Molloy and Ford(2017)]{Molloy2017}
Molloy, Timothy~L, and Jason~J Ford. 2017. ``Misspecified and asymptotically
  minimax robust quickest change detection.'' \emph{IEEE Transactions on Signal
  Processing} 65 (21): 5730--5742.

\bibitem[Moulin and Veeravalli(2018)]{moulin-veeravalli-2018}
Moulin, Pierre, and Venugopal~V Veeravalli. 2018. \emph{Statistical Inference
  for Engineers and Data Scientists}. Cambridge University Press.

\bibitem[Moustakides(1986)]{moustakides1986optimal}
Moustakides, George~V. 1986. ``Optimal stopping times for detecting changes in
  distributions.'' \emph{Annals of Statistics} 14 (4): 1379--1387.

\bibitem[Moustakides and Basioti(2019)]{moustakides2019training}
Moustakides, George~V, and Kalliopi Basioti. 2019. ``Training neural networks
  for likelihood/density ratio estimation.''
  \urlprefix\url{https://arxiv.org/abs/1911.00405}.

\bibitem[Mukhopadhyay(2014)]{mukhopadhyay2014wearable}
Mukhopadhyay, Subhas~Chandra. 2014. ``Wearable sensors for human activity
  monitoring: A review.'' \emph{IEEE sensors journal} 15 (3): 1321--1330.

\bibitem[Page(1954)]{page-biometrica-1954}
Page, E.~S. 1954. ``Continuous inspection schemes.'' \emph{Biometrika} 41
  (1/2): 100--115.

\bibitem[Pollak(1978)]{pollakmixture}
Pollak, Moshe. 1978. ``Optimality and almost optimality of mixture stopping
  rules.'' \emph{Annals of Statistics} 6 (4): 910--916.

\bibitem[Pollak(1985)]{poll-astat-1985}
Pollak, Moshe. 1985. ``Optimal detection of a change in distribution.''
  \emph{Annals of Statistics} 13 (1): 206--227.

\bibitem[Poor and Hadjiliadis(2008)]{poor2008quickest}
Poor, H~Vincent, and Olympia Hadjiliadis. 2008. \emph{Quickest Detection}.
  Cambridge University Press.

\bibitem[Raginsky and Sason(2013)]{raginsky2013concentration}
Raginsky, Maxim, and Igal Sason. 2013. ``Concentration of measure inequalities
  in information theory, communications, and coding.'' \emph{Foundations and
  Trends{\textregistered} in Communications and Information Theory} 10 (1-2):
  1--246.

\bibitem[Siegmund(1985)]{siegmund1985sequential}
Siegmund, David. 1985. \emph{Sequential Analysis: Tests and Confidence
  Intervals}. Springer Series in Statistics. Springer.

\bibitem[Sun and Zou(2021)]{sun2021data}
Sun, Zhongchang, and Shaofeng Zou. 2021. ``A data-driven approach to robust
  hypothesis testing using kernel {MMD} uncertainty sets.'' In \emph{2021 IEEE
  International Symposium on Information Theory (ISIT)}, 3056--3061. IEEE.

\bibitem[Tartakovsky, Nikiforov, and
  Basseville(2015)]{tartakovsky2014sequential}
Tartakovsky, Alexander, Igor Nikiforov, and Mich{\`e}le Basseville. 2015.
  \emph{Sequential Analysis: Hypothesis Testing and Changepoint Detection}.
  ser. Monographs on Statistics and Applied Probability 136. Boca Raton,
  London, New York: Chapman \& Hall/CRC Press, Taylor \& Francis Group.

\bibitem[Unnikrishnan, Veeravalli, and Meyn(2011)]{Unnikrishnan2011}
Unnikrishnan, Jayakrishnan, Venugopal~V Veeravalli, and Sean~P Meyn. 2011.
  ``Minimax robust quickest change detection.'' \emph{IEEE Transactions on
  Information Theory} 57 (3): 1604--1614.

\bibitem[Veeravalli and Banerjee(2013)]{veeravalli2013quickest}
Veeravalli, Venugopal~V, and Taposh Banerjee. 2013. ``Quickest change
  detection.'' \emph{Academic Press Library in Signal Processing: {A}rray and
  Statistical Signal Processing} 3: 209--256.

\bibitem[Villani(2003)]{villani2003topics}
Villani, C{\'e}dric. 2003. \emph{Topics in Optimal Transportation}. American
  Mathematical Society.

\bibitem[Wang and Xie(2022)]{wang2022data}
Wang, Jie, and Yao Xie. 2022. ``A data-driven approach to robust hypothesis
  testing using {S}inkhorn uncertainty sets.'' In \emph{Proceedings of the IEEE
  International Symposium on Information Theory (ISIT)}, 3315--3320. IEEE.

\bibitem[Wasserman(2006)]{wasserman2006all-nonpara-stat}
Wasserman, Larry. 2006. \emph{All of Nonparametric Statistics}. 233 Spring
  Street, New York, NY 10013, USA: Springer.

\bibitem[Xie et~al.(2021)]{xie2021sequential}
Xie, Liyan, Shaofeng Zou, Yao Xie, and Venugopal~V Veeravalli. 2021.
  ``Sequential (quickest) change detection: Classical results and new
  directions.'' \emph{IEEE Journal on Selected Areas in Information Theory} 2
  (2): 494--514.

\end{thebibliography}


\appendix

\section*{Appendix A: A Useful Lemma for Proof of Theorem\,\ref{thm:lfd}}
\refstepcounter{section}\label{sec:proof_main}

We present the following lemma which is used in the proof of Theorem\,\ref{thm:lfd}. 
\begin{customlem}{A.1}\label{lem:opt_lem}
    Given constants $c_1,c_2,\ldots,c_n\in\mathbb{R}$, consider the following optimization problem:
    \begin{equation}
        \label{eq:lem5-obj}
     \min_{a_1,a_2,\ldots,a_n\geq0} f(a_1,a_2,\ldots,a_n):= \Big(\sum_{i=1}^n a_i\Big) \log\Big(\sum_{i=1}^n a_i\Big) + \sum_{i=1}^n c_i a_i,
    \end{equation}
    where by convention, we let $0\log 0=0$. Then, the minimizer $(a_1^*,\ldots,a_n^*)$ satisfies
    \[
    \sum_{i=1}^n a_i^*=e^{-\min_{i=1,\ldots,n} c_i - 1},
    \]
    and the optimal objective value is $f(a_1^*,a_2^*,\ldots,a_n^*) = - e^{-\min_{i=1,\ldots,n} c_i - 1}$.
\end{customlem}
\begin{proof}
    We first note that it suffices to consider the case where all the $c_i$'s are distinct, i.e., $c_i\neq c_j$, $\forall i\neq j$. This is due to the fact that if any pair of $c_i$ and $c_j$ are equal, we can re-define our set of constants to be the unique values in $\{c_1,\ldots,c_n\}$, let's denote these as $\{c_1',\ldots,c_k'\}$ ($k<n$). We then define new variables $\Tilde{a}_{i} = \sum_{j:c_j=c_i'} a_j$, $i=1,\ldots,k$. The problem \eqref{eq:lem5-obj} thus becomes equivalent to:
    \[
    \min_{\Tilde{a}_{1},\ldots,\Tilde{a}_{k}\geq0} f(\Tilde{a}_{1},\ldots,\Tilde{a}_{k}):= (\sum_{i=1}^k \Tilde{a}_{i}) \log(\sum_{i=1}^k \Tilde{a}_{i}) + \sum_{i=1}^k c_i' \Tilde{a}_{i},
    \]
    with distinct $c_i'$ values, and this new problem has the same optimal value as the original problem in \eqref{eq:lem5-obj}. Therefore, we can safely assume that $c_1,\ldots,c_n$ are different from each other for the remaining proof, i.e., $c_i\neq c_j$, $\forall i\neq j$.
    
    We introduce Lagrangian multipliers $\lambda_i\geq 0$ for $i=1,\ldots,n$ for the constraints. The corresponding Lagrangian function is then given by
    \[    L(a_1,\ldots,a_n,\lambda_1,\ldots,\lambda_n) = \Big(\sum_{i=1}^n a_i\Big) \log\Big(\sum_{i=1}^n a_i\Big) + \sum_{i=1}^n c_i a_i - \sum_{i=1}^n \lambda_i a_i.
    \]
    By applying the Karush–Kuhn–Tucker condition, the optimal solution $(a_1^*,\ldots,a_n^*)$ must satisfy the gradient condition
    \begin{equation}
        \label{eq:lem5-eq1}
        \frac{\partial L}{\partial a_i^*} = 1 + \log\Big(\sum_{i=1}^n a_i^*\Big) + c_i - \lambda_i=0,  \ \forall i=1,2,\ldots,n,
    \end{equation}
    and the complementary slackness conditions
    \begin{equation}
        \label{eq:lem5-eq2}
    \lambda_i a_i^* = 0, \ \forall i=1,2,\ldots,n.
    \end{equation}    
    From \eqref{eq:lem5-eq1}, we deduce that  
    \[
    \lambda_i = 1 + \log\Big(\sum_{i=1}^n a_i^*\Big) + c_i, \ \forall i=1,2,\ldots,n,
    \]
    which implies that $\lambda_1,\ldots,\lambda_n$ are distinct. Now, we consider two scenarios: 
    
    \noindent (i) If $\lambda_i\neq 0$, $\forall i$, then from \eqref{eq:lem5-eq2} we get $a_1^*=a_2^*=\cdots=a_n^*=0$ and the objective value is zero.
    
    \noindent (ii) If there exists $i_0$ such that $\lambda_{i_0}=0$, then from $\lambda_i\geq 0$, $\forall i$, we have that 
    \[
    i_0=\argmin \lambda_i = \argmin c_i.
    \]
    Additionally, by \eqref{eq:lem5-eq2}, we have $a_j^*=0$ for $j\neq i_0$ since $\lambda_j \neq \lambda_{i_0}=0$, and the corresponding $a_{i_0}^*=e^{-c_{i_0}-1}$ from \eqref{eq:lem5-eq1}, yielding an objective value of $-e^{-c_{i_0}-1}<0$. 
    
    Therefore, when $c_1,\ldots,c_n$ are distinct and $i_0=\argmin c_i$, the minimizer is given by $a_{i_0}^*=e^{-c_{i_0}-1}$ and $a_{j}^*=0$, $\forall j\neq i_0$, and the optimal value is $-e^{-c_{i_0}-1}$.  
    In summary, the optimal solution $(a_1^*,\ldots,a_n^*)$ to \eqref{eq:lem5-obj} satisfies $\sum_{i=1}^n a_i^*=e^{-\min_{i=1,\ldots,n} c_i - 1}$, and the corresponding optimal objective value is $-e^{-\min_{i=1,\ldots,n} c_i - 1}$. 
\end{proof}

\section*{Appendix B: More about the Transportation-Cost Inequality}
\refstepcounter{section}\label{sec:app_T1T2}

In this section, we present example distributions that satisfy such a Transportation-Cost Inequality in Definition\,\ref{def:T_ineq}, to demonstrate the wide applicability of the results in Section\,\ref{sec:radius_select}. 

\paragraph{Examples of $T_1(c)$ inequality:}  For a discrete sample space with the Hamming metric $c(x,y) = 1_{\{x \neq y\}}$, the $W_1$ distance satisfies the following inequality 
\[
\wass_1(P,Q) = \| P - Q \|_{\mathrm{TV}}\leq \sqrt{\frac12 \mathsf{KL}(Q||P)},
\]
which is a consequence of Pinsker's inequality. Hence, the $T_1(1/4)$ inequality holds for every probability measure on the discrete sample space.

\paragraph{Examples of $T_2(c)$ inequality:} For $\mathcal X=\mathbb R^n$ and $c(x,y)=\| x-y\|_2$, the standard $n$-dimensional Gaussian distribution satisfies the $T_2(1)$ inequality, i.e., $\wass_2(P,Q) \leq \sqrt{2 \mathsf{KL}(Q||P)}$ for $P$ being the $n$-dimensional Gaussian distribution and $Q$ being any distribution satisfying $Q \ll P$. More generally, if $P=N(\mu,\Sigma)$, where $\mu$ is the mean vector and $\Sigma$ is the covariance matrix, then $P$ satisfies the $T_2(c)$ inequality for $c=1/2\kappa$, where $\kappa\leq \min_i \lambda_i(\Sigma^{-1})$, representing the smallest eigenvalue of the inverse covariance matrix.

\section*{Appendix C: Proofs for Section\,\ref{sec:multiple_pcs}}
\refstepcounter{section}\label{sec:app_proof5}

\begin{proof}[Proof of Lemma\,\ref{lemma:fa}]
  We define the following quantity
\[ Z_j^{(m)} := \log \frac{p_{(m)}^\ast(X_j)}{q(X_j)}. \]
First, we can assume that $\bE_\infty [\tau_{\mathrm{DR}}(b)] < \infty$, as otherwise the statement would be trivial. Consider the following test based on the Shiryaev-Roberts (SR) statistic,
\[ 
\tau_b^R := \inf\cbrc{k \in \bN: \sum_{m=1}^M \sum_{n=1}^k \prod_{j=n}^k e^{Z^{(m)}_j} =: \sum_{m=1}^M R^{(m)}_k \geq e^b}. 
\]
Note that $\tau_b^R \leq \tau_{\mathrm{DR}}(b)$ since $\sum_{m=1}^M R^{(m)}_k \geq \max_{m=1,\ldots,M} S_k^{(m)}$. Hence, $\bE_\infty \sbrc{\tau^R_b} < \infty$. Denoting $R_k := \sum_{m=1}^M R^{(m)}_k$, we have
\begin{align*}
    \bE_\infty \sbrc{R_k | \mathcal{F}_{k-1}} &= \sum_{m=1}^M \bE_\infty \sbrc{(1+R^{(m)}_{k-1}) e^{Z^{(m)}_k} | \mathcal{F}_{k-1}}\\
    &= \sum_{m=1}^M (1+R^{(m)}_{k-1}) = M + R_{k-1},
\end{align*}
which implies that the sequence $\{ (R_k - M k) \}_{k \geq 1}$ forms a martingale.
Furthermore, since $R_k \in (0,e^b)$ almost surely on the event $\{\tau^R_b > k\}$, we have for any $k \geq 1$,
\begin{align*}
    \bE_\infty \sbrc{\big|(R_{k+1} - M (k+1)) - (R_k - M k)\big|\Big|\mathcal{F}_{k}} 
    & = \bE_\infty \sbrc{\big|R_{k+1} - R_k - M\big|\Big|\mathcal{F}_{k}}\\
    &\leq \bE_\infty \sbrc{R_{k+1}|\mathcal{F}_{k}} + (R_k + M)\\
    &= 2 (R_k + M)\\
    &\leq 2 (e^b + M)
\end{align*}
almost surely on the event $\{\tau^R_b > k\}$. Therefore, by the Optional Stopping Theorem,
\[
\bE_\infty [R_{\tau_b^R}] = M \bE_\infty [\tau_b^R]. 
\]
Since $R_{\tau_b^R}\geq e^b$, it follows that $\bE_\infty [\tau_b^R] \geq M^{-1} e^b$ and consequently 
\[
\bE_\infty [\tau_{\mathrm{DR}}(b)]\geq \bE_\infty [\tau_b^R]\geq M^{-1} e^b.
\]
\end{proof}


\begin{proof}[Proof to Theorem\,\ref{thm:opt}]
    Let $\calP_m := \calP_{n_m}^{(m)}$. Recall that $P_{(m)}^*$ represents the LFD for the $m$-th class and that $p_{(m)}^*$ denotes its density with respect to the dominating measure $\mu$. For the lower bound, we have
\begin{align*}
  \inf_{\tau' \in C(\gamma)} \max_{i=1,\ldots,M} \sup_{P_i\in\calP_{i}} \WADD^{P_i}(\tau') 
 & \geq \max_{i=1,\ldots,M} \sup_{P_i\in\calP_{i}} \inf_{\tau' \in C(\gamma)} \WADD^{P_i}(\tau') \\
 & \geq \frac{\log\gamma}{I^*} (1 + o(1)),    
\end{align*}
where the last inequality follows from \cite[Thm.~1]{lai-ieeetit-1998} and the weak law of large numbers for independent random variables.

For the upper bound, we consider the detection rule $\tau_{\mathrm{DR}}(b)$ defined in \eqref{eq:CuSum_stop_multi}. For any distribution $P_i \in \mathcal P_i$ we have 
\[
\begin{aligned}
& \mathbb E^{P_i}\sbrc{\log \frac{p_{(i)}^*(X)}{q(X)}} 
= \mathsf{KL}(P_i||Q) - \mathsf{KL}(P_i||P_{(i)}^*) \geq \mathsf{KL}(P_{(i)}^*||Q),
\end{aligned}
\]
where the last inequality is a consequence of the weak stochastic boundedness condition. Therefore, according to \cite[Thm.~4(ii)]{lai-ieeetit-1998}, we have
\[
\max_{i=1,\ldots,M}\sup_{P_i\in\calP_{i}}\WADD^{P_i}(\tau_{\mathrm{DR}}(b)) \leq \frac{b}{I^*} (1 + o(1)). 
\]
By selecting $b = b_\gamma = \log \gamma + \log M$, we obtain
\begin{align*}
\max_{i=1,\ldots,M}\sup_{P_i\in\calP_i}\WADD^{P_i}[\tau_{\mathrm{DR}}(b_\gamma)] &\leq \frac{\log \gamma + \log M}{I^*} (1 + o(1))\\
&= \frac{ \log \gamma}{I^*} (1 + o(1)).
\end{align*}
Finally, from Lemma~\ref{lemma:fa}, the proof is complete since $\tau_{\mathrm{DR}}(b_\gamma) \in C(\gamma)$ when $b = b_\gamma$.
\end{proof}

\section*{Appendix D: More Numerical Results --  Effect of Radius}
\refstepcounter{section}\label{sec:numerical_appendix}

In this section, we investigate the influence of the radius of the uncertainty sets. We simplify the analysis by considering the case where $M=1$. We examine two distribution models: Gaussian and Beta. In both scenarios, we collect a set of 10 samples as training data for the post-change distribution.

\begin{figure}[htbp!]
    \centering
    \begin{tabular}{cc}
\includegraphics[width=0.4\linewidth]{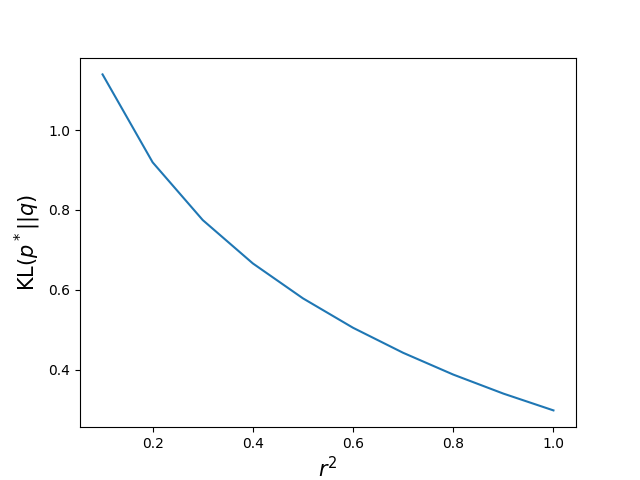} & \includegraphics[width=0.4\linewidth]{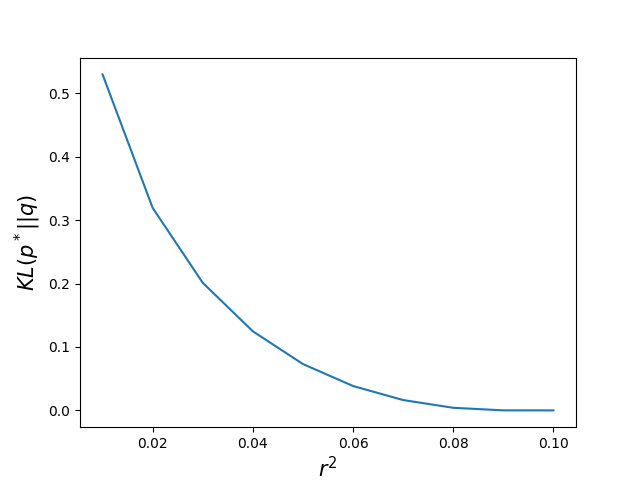}
    \end{tabular}
      \caption{KL-divergence as a function of radius. In the Gaussian example (left), we consider the pre-change distribution as ${\cal N}(0,1)$ and the post-change distribution as ${\cal N}(1,1)$. In the Beta example (right), we consider the pre-change distribution as Beta(2, 3) and the post-change distribution as Beta(3, 
 2). Both examples are evaluated with a Wasserstein order $s=2$.}
    \label{fig:kldiv_r}
\end{figure}

In Fig.~\ref{fig:kldiv_r}, it is evident that the smallest KL divergence over the post-change uncertainty set, which represents the divergence between the post-change LFD and the pre-change distribution, decreases as the radius increases. When the radius $r$ surpasses a certain threshold, this smallest KL divergence becomes zero, indicating that the post-change uncertainty set encompasses the pre-change distribution. In practical applications, since this smallest KL divergence can be pre-computed using the pre-change distribution and the post-change training samples, one can choose the radius based on a desired threshold for this value. Specifically, suppose we are only interested in detecting changes from $Q$ to $P$ with a minimum KL divergence of $\mathsf{KL}(P||Q) \geq I^\text{min}$, where $I^\text{min}>0$. In that case, the radius can be selected such that the KL divergence between the post-change LFD and the pre-change distribution is greater than $I^\text{min}$.

\end{document}